\documentclass{amsart}
\usepackage{amsmath}
\usepackage{amssymb}
\usepackage[utf8]{inputenc}
\usepackage{enumerate} 
\usepackage{graphicx}
\usepackage[outdir=./]{epstopdf}
\usepackage{enumitem}

\newtheorem{thm}{Theorem}[section]

\newtheorem{lem}[thm]{Lemma}
\newtheorem{cor}[thm]{Corollary}
\newtheorem{conj}{Conjecture}

\theoremstyle{definition}
\newtheorem{definition}[thm]{Definition}

\theoremstyle{remark}

\newtheorem{remark}[thm]{Remark}

\numberwithin{equation}{section}

\newcommand{\R}{\mathbb{R}}  
\newcommand{\N}{\mathbb{N}}

\newcommand{\norm}[1]{\left\lVert#1\right\rVert}

 {\begin{list}{}%
         {\setlength{\leftmargin}{#1}}%
         \item[]%
 }
 {\end{list}}

\newenvironment{claim}[1]{\par\noindent\underline{Claim:}\space#1}{}
\newenvironment{claimproof}[1]{\par\noindent\underline{Proof of Claim:}\space#1}{\hfill $\blacksquare$}

\begin{document}

\title[A family of non-collapsed steady Ricci Solitons]{A family of non-collapsed steady Ricci Solitons in even dimensions greater or equal to four}
\author{Alexander Appleton}
\address{Department of Mathematics, UC Berkeley, 
CA 94720, USA}
\email{aja44@berkeley.edu}
\maketitle

\begin{abstract}
We construct a family of non-collapsed, non-K\"ahler, non-Einstein steady gradient Ricci solitons in even dimensions greater or equal to four. These solitons are diffeomorphic to the total space of certain complex line bundles over K\"ahler-Einstein manifolds of positive scalar curvature. In four-dimensions this leads to a family of $U(2)$-invariant non-collapsed steady gradient Ricci solitons on the total spaces of the line bundles $O(-k)$, $k \geq 3$, over $\mathbb{C}P^1$. As a byproduct of our methods we also construct Taub-Nut like Ricci solitons and demonstrate a new proof for the existence of the Bryant soliton in the final part of our paper.
\end{abstract}

\tableofcontents

\section{Introduction}
In this paper we construct new families of non-collapsed and non-K\"ahler steady gradient Ricci solitons in even dimensions greater or equal to four. A Ricci soliton $(M,g)$ is a self-similar solution to the Ricci flow equations 
\begin{equation}
\label{ricciflow}
\partial_t g_{ij} = -2\mathrm{Ric}_{ij} 
\end{equation}
that, up to diffeomorphism, homothetically shrinks, expands, or remains steady under Ricci flow. We will only study steady gradient solitons which satisfy the equation
\begin{equation}
\label{soliton}
\mathrm{Ric}_{ij} + \nabla_i \nabla_j f = 0,
\end{equation}
for a smooth potential function $f: M \rightarrow \R$. Solitons are interesting objects, because they are candidates for blow-up limits of singularities in Ricci Flow. In particular, Type I singularities correspond to shrinking solitons and all Type II singularities known so far are modeled on steady solitons. The new non-collapsed steady solitons we find here are likely to occur as singularity models in Ricci flow. In a separate paper, still in preparation, we have conducted numerical simulations verifying this. 

In three dimensions the classification of non-expanding solitons has largely been carried out and singularity formation is well-understood. But in four dimensions, even though finding non-expanding Ricci solitons is a fundamental problem, to date there are surprisingly few examples known. The last one was discovered by Feldman, Ilmanen, and Knopf \cite{FIK03} --- the FIK shrinker --- also shown to occur as a singularity model by Maximo \cite{M14}. Before this, Cao \cite{Cao96} had constructed a $U(2)$-invariant steady K\"ahler-Ricci soliton. This soliton, however, is collapsed and hence, as shown in Perelman's work \cite{Perl08}, does not appear as a blow-up limit. The rotationally symmetric Bryant soliton \cite{B05} is the last non-collapsed, non-K\"ahler, non-expanding soliton discovered in four dimensions. 

The four dimensional solitons constructed in this paper are asymptotic to the Bryant soliton's quotient by a cyclic group $\mathbb{Z}_k$ of order $k \geq 3$, and their underlying manifold is diffeomorphic to the completion of $\mathbb{R}_{>0} \times S^3/\mathbb{Z}_k$, $k\geq 3$, obtained by adding an $S^2$ at the origin.  Relying on an idea of Page and Pope \cite{PP87}, our methods carry over to complex line bundles over K\"ahler-Einstein manifolds of positive scalar curvature. This allows us to prove the existence of non-collapsed steady solitons on such bundles, supposing their degrees are sufficiently large. In doing so we obtain $(2n+2)$-dimensional solitons on the line bundles $O(-k)$, $k\geq n+1$, over $\mathbb{C}P^n$, which are also asymptotic to a quotient of the $(2n+2)$-dimensional Bryant soliton.

For the metrics considered in this paper the Ricci soliton equation (\ref{soliton}) reduces to a system of ordinary differential equations. Our main result is showing that for a critical choice of boundary conditions the ODE yields a non-collapsed soliton. As an intermediate step we prove the existence of a 1-parameter family of complete \emph{collapsed} solitons. These were independently and by different methods discovered in \cite{Wink17} and \cite{Stol17}. In the final part of the paper we apply our methods to $U(n)$-invariant metrics on $\R^{2n}$, $n\geq2$. This allows us to construct a new 1-parameter family of Taub-Nut like Ricci solitons and also yields an alternative derivation of the Bryant soliton in even dimensions greater or equal to four. 

\subsection{Some background on solitons and Ricci flow singularities}
Solitons are important objects in the study of Ricci flow, because they arise as blow-up limits of singularities. Let $g(t)$, $t\in[0,T)$, be a Ricci flow on a closed manifold $M$ which develops a singularity at time $T < \infty$. Then there exists a point $p \in M$ and a sequence of times $t_i \rightarrow T$ such that
$$K_i=|\mathrm{Rm}_{g(t_i)}|(p) \rightarrow \infty \quad \text{as} \quad t_i \rightarrow T.$$ 
By Perelman's work the sequence of parabolically dilated Ricci flows  
\begin{equation*}
g_i(t):= K_i g\left(t_i + K_i^{-1}t\right), \quad t \in [-K_i t_i, 0],
\end{equation*}
converges in a suitable sense to an ancient solution $(M_{\infty}, g_{\infty}(t)), t\leq 0,$ of the Ricci Flow \cite[Theorem 6.68]{ChI}. This limiting Ricci flow is called the singularity model. Note that the manifold $M_{\infty}$ need not be diffeomorphic to $M$.

It is useful to distinguish between Type I and Type II singularities (see \cite[Section 16]{Ham95}), defined by the rate at which the curvature diverges: 

\begin{align*}
\text{Type I singularity:} \quad & \sup\:(T-t_i)K_i < \infty \\
\text{Type II singularity:} \quad & \sup\: (T-t_i)K_i = \infty \\
\end{align*}
Type I singularities are modeled on shrinking Ricci solitons, as was shown in the work of \cite{N10}, \cite{EMT11}. However, for Type II singularities less is known --- even though steady Ricci solitons are natural candidates and currently all known examples are modeled on them \cite{GZ08}, \cite{AIK11}, \cite{W14}. On the other hand certain solitons cannot occur as singularity models. In particular, only \emph{non-collapsed} solitons can be singularity models due to Perelman's no local collapsing theorem \cite[Section 4]{Perl08}:
\begin{definition}
A Riemannian manifold $(M,g)$ is $\kappa$-non-collapsed below the scale $r>0$ at the point $x$ if $|\mathrm{Rm}(g)| \leq r^{-2}$ for all $y \in B(x,r)$ and 
\begin{equation*}
\frac{\text{Vol} B(x,r)}{r^n} \geq \kappa.
\end{equation*}
A soliton is non-collapsed if for some $\kappa>0$ it is $\kappa$-non-collapsed at all points and scales.
\end{definition}
A-priori the steady Ricci solitons constructed in this paper may arise as singularity models, as they are non-collapsed. Preliminary Ricci flow simulations carried out by the author in collaboration with Jon Wilkening indicate that they in fact do.  

\subsection{Steady solitons in four dimensions}
In four dimensions the topology and geometry of the Ricci solitons constructed are easy to describe: They are diffeomorphic to the complex line bundles $O(-k), k \geq 0,$ over $\mathbb{C}P^1 \cong S^2$. For our purposes it is useful to consider these manifolds as the completion of $\R_{>0} \times S^3/\mathbb{Z}_k$, $k \in \N$, by adding an $S^2$ at the origin. We equip these manifolds with a $U(2)$-invariant metric, which away from the central $S^2$ can be written as a warped-product metric of the form  
\begin{equation}
\label{U2-inv-metric} 
g = ds^2 + g_{a(s),b(s)}.
\end{equation} 
Here $a,b : (0,\infty) \rightarrow \R$ are functions of $s$ and $g_{a(s),b(s)}$ are squashed Berger metrics on the cross-section $S^3/\mathbb{Z}_k$. For metrics of this form the soliton equation (\ref{soliton}) reduces to a system of ordinary differential equations for $a$, $b$ and $f$.
 
Below we describe Berger metrics in more detail. For this, recall the Hopf fibration $\pi: S^3 \rightarrow S^2$, which arises from the multiplicative action of the unitary group
\begin{equation*}
U(1) = \{ e^{i \theta} \; \big | \; \theta \in [0, 2\pi) \} \cong S^1
\end{equation*} on 
\begin{equation*}
S^3 \cong \{ (z_1, z_2) \in \mathbb{C}^2 \; \big |  \; |z_1|^2 + |z_2|^2 = 1 \} \subset \mathbb{C}^2.
\end{equation*}
When $S^3$ and $S^2$ are equipped with the round metrics of curvatures $1$ and $4$, respectively, $U(1)$ acts by isometries and the quotient map $\pi$ is a Riemannian submersion. Thus the round metric of curvature 1 on $S^3$ can be written as
\begin{equation}
\label{warped}
g = \sigma \otimes \sigma + \pi^{\ast}g_{S^2(\frac{1}{2})},
\end{equation}
where the one-form $\sigma$ is dual to the vertical $S^1$-fiber directions and $g_{S^2(\frac{1}{2})}$ is the round metric of curvature 4 on $S^2$. Rescaling the vertical and horizontal directions by factors $a>0$ and $b>0$, respectively, yields the squashed Berger metric
\begin{equation*}
g_{a,b} = a^2 \sigma \otimes \sigma +  b^2 \pi^{\ast}g_{S^2(\frac{1}{2})}
\end{equation*}
on $S^3$, which are also invariant under the $U(1)$-action above. The cross-sectional metrics $g_{a(s),b(s)}, s>0,$ of the warped product metric (\ref{U2-inv-metric}) arise from the quotient of the Berger metric by the cyclical subgroup 
\begin{equation*} 
\mathbb{Z}_k = \{ e^{2\pi i \frac{l}{k}} \; \big | \; l = 0, 1, \cdots, k-1\} \subset U(1).
\end{equation*}  
We extend the metric (\ref{U2-inv-metric}) across the central $S^2$ by taking $a(0) = 0$ and $b(0) > 0$. Geometrically this means that 
\begin{enumerate}
\item the metric pulls back to the round metric of curvature $\frac{b^2}{4}$ on the central $S^2$
\item the $S^1$-fibers of $S^3/\mathbb{Z}_k$ shrink to a point on the central $S^2$ as $s \rightarrow 0$
\end{enumerate}
Because the $S^1$-fibers of $S^3/\mathbb{Z}_k$ are parameterized by $\theta \in [0,\frac{2\pi}{k})$, their circumferences are equal to $\frac{2\pi}{k}a(s)$ and behave like $\frac{2\pi}{k} a'(0)s + O(s^2)$ as $s \rightarrow 0$. Therefore we must require $a'(0) = k$ to avoid a conical singularity at $s=0$. This is how the topology of the manifolds enters the analysis of the Ricci soliton equation.

\vspace{1em}\noindent
Many important metrics are of the form (\ref{U2-inv-metric}), in particular
\begin{itemize}
\item (k = 1): The K\"ahler-Ricci FIK shrinker \cite{FIK03}
\item (k = 1): The Ricci-flat Taub-Bolt metric \cite{P78}
\item (k = 2): The asymptotically locally Euclidean (ALE) Ricci-flat Eguchi-Hanson metric \cite{EH79} 
\end{itemize}
We show that when $k \geq 3$ there exists a non-collapsed steady Ricci soliton:


\begin{thm}[4d non-collapsed steady Ricci solitons]
\label{cor:4d-non-collapsed-solitons}
When $k \geq 3$ there exists a complete non-collapsed steady gradient Ricci soliton on the completion of $\R_{>0} \times S^3/\mathbb{Z}_k$ --- by adding in an $S^2$ at the origin --- equipped with a U(2)-invariant metric of the form (\ref{U2-inv-metric}). These solitons are diffeomorphic to the total space of the complex line bundle $O(-k)$ over $\mathbb{C}P^1$. Moreover, they satisfy
$$a \sim b \sim C \sqrt{s} \quad \text{as} \quad s\rightarrow \infty$$ 
for $C>0$ a constant, and are therefore asymptotic to the quotient of the 4d Bryant soliton \cite{B05} by $\mathbb{Z}_k$.
\end{thm}

\subsection{Steady solitons on line bundles over $\mathbb{C}P^n$}
We also construct $2n+2$, $n \geq 2$, dimensional steady Ricci solitons within the family of $U(n+1)$-invariant metrics on the total space of the complex 
line bundles $O(-k), k\in\N,$ over $\mathbb{C}P^n, n \in \N$. These spaces can be viewed as warped product
metrics of the form
\begin{equation}
\label{warped_hd}
g = ds^2 + g_{a(s),b(s)} = ds^2 + a(s)^2 \sigma \otimes \sigma + b(s)^2 \pi^{\ast}g_{\mathbb{C}P^n}.
\end{equation}
on the dense subset $\R_{>0} \times S^{2n+1}/\mathbb{Z}_k \subset O(-k), n, k \in \N$, which are smoothly completed across the central $\mathbb{C}P^n$ at the origin. Here the metric $g_{a(s),b(s)}$ is defined analogously to the 4d case via the Hopf fibration $\pi: S^{2n+1} \rightarrow \mathbb{C}P^n$ and $g_{\mathbb{C}P^n}$ is the Fubini-Study metric. Because of the $U(n+1)$-symmetry, the Ricci soliton equation (\ref{soliton})
reduces, up to changes in multiplicative constants, to the same system of linear differential equations as in the four dimensional case above.
This allows us to generalize Theorem \ref{cor:4d-non-collapsed-solitons} to 

\begin{thm}[non-collapsed steady Ricci solitons on line bundles over $\mathbb{C}P^n$]
\label{cor:CPn-solitons}
When $k > n+1$ there exists a complete non-collapsed steady gradient Ricci soliton on the completion of $\R_{>0} \times S^{2n+1}/\mathbb{Z}_k$ --- by adding in an $\mathbb{C}P^n$ at the origin --- equipped with a $U(n+1)$-invariant metric of the form (\ref{warped_hd}). These solitons are diffeomorphic to the total space of the complex line bundle $O(-k)$ over $\mathbb{C}P^n$. Moreover, they satisfy
$$a \sim b \sim C \sqrt{s} \quad \text{as} \quad s\rightarrow \infty$$ 
for $C>0$ a constant, and are therefore asymptotic to the quotient of the $(2n+2)$-dimensional Bryant soliton \cite{B05} by $\mathbb{Z}_k$.
\end{thm}

\subsection{Steady solitons on line bundles over K\"ahler-Einstein manifolds}
Relying on ideas developed in \cite{BB85} and \cite{PP87}, our methods generalize further to a class of metrics on complex line bundles over K\"ahler-Einstein manifolds of positive scalar curvature. We describe these manifolds here: Let $M$ denote the total space of a complex line bundle over a K\"ahler-Einstein manifold $(\hat{M},J,\hat{g})$ of positive scalar curvature. Let $\omega$ and $\rho$ be the K\"ahler and Ricci forms, respectively, and assume that the metric $\hat{g}$ is scaled such that $\rho = 2(n+1)\omega$. Then $\frac{\rho}{2\pi} \in H^2(\hat{M},\mathbb{Z})$ is the Chern class of the canonical line bundle over $\hat{M}$ and therefore integral. Thus there exists a $p=p(\hat{M}, \omega) \in \N$ such that $\frac{\rho}{2\pi} = p \sigma$, where 
$\sigma$ generates the cohomology group $H^2(\hat{M},\mathbb{Z})$. Denote by $L_{k}$ the total space of the complex line bundle with Chern class equal to $k \sigma, k \in \N,$ over $\hat{M}$. We equip $L_k$ with a metric of the form
\begin{equation}
\label{metric}
g = ds^2 + a(s)^2 \left(d\tau - 2A\right)^2 + b(s)^2 \hat{g},
\end{equation}
where 
\begin{itemize}
\item $A$ is a connection 1-form satisfying $dA = \omega$ on $\hat{M}$ 
\item $\tau \in [0, 2\pi)$ is an angular coordinate of the $S^1$ subbundle of $L_{k}$ 
\item $s$ is the radial coordinate. 
\end{itemize}

\begin{remark}
In the case of $\hat{M} = \mathbb{C}P^n$ equipped with the Fubini-Study metric, the spaces $L_k$ are diffeomorphic to the line complex bundles $O(-k)$ over
$\mathbb{C}P^n$. in addition to this the warped product metrics (\ref{warped}) can be written in the form (\ref{metric}), where the functions $a(s)$ and $b(s)$ retain same geometrical interpretation and the 1-form $d\tau -2 A$ corresponds to $\sigma$.
\end{remark}

We prove that when $k > p(\hat{M}, \omega)$ there exists a steady gradient Ricci soliton on $L_k$, yielding the following theorem:
\begin{thm}
\label{main-thm}
Let $(\hat{M}, J, \hat{g})$ be a K\"ahler-Einstein manifold of positive scalar curvature. Then there exists a non-collapsed steady gradient Ricci soliton on $L_{k}$ when $k > p(\hat{M}, \omega)$. For these solitons $a \sim b \sim C \sqrt{s}$ as $s\rightarrow \infty$, where $C>0$ a constant.
\end{thm}

\subsection{Taub-Nut like solitons}

In the final part of this paper we construct Taub-Nut like, $U(n+1)$-invariant steady Ricci solitons on $\mathbb{R}^{2n+2}, n \geq 1,$ and give another proof of the existence of the Bryant soliton in even dimensions greater or equal to four. As $U(n+1)$-invariant metrics on $\R^{2n+2}$ can be written in 
the form of the warped product metric (\ref{warped_hd}) on the dense subset $\R_{>0} \times S^{2n+1} \subset \mathbb{R}^{2n+2}$, the Ricci soliton equation (\ref{soliton}) reduces to the same system of linear differential equations. We merely need to modify the boundary conditions to account for the change in topology. In particular, we need to require $a=b=0$ and $a'=b'=1$ at $s=0$. The Taub-Nut metrics (see \cite{T51}, \cite{H77} and \cite{BB85}) are of this form. Notice also that when $a=b$ everywhere the resulting metric is rotationally symmetric. This is exploited in the construction of the Bryant soliton.

\section{Gradient steady Ricci soliton equations}
In Appendix A we show how for a metric of the form (\ref{metric}) the steady gradient Ricci soliton equation (\ref{soliton}) reduces to the following system of ODEs 
\begin{align}
\label{solitoneqn1}
f'' &= \frac{a''}{a} + 2n \frac{b''}{b} \\
\label{solitoneqn2}
a''&= 2n\left(\frac{a^3}{b^4}- \frac{a'b'}{b}\right) + a'f' \\
\label{solitoneqn3}
b''&= \frac{2n+2}{b} - 2 \frac{a^2}{b^3}-\frac{a'b'}{a} - (2n-1)\frac{(b')^2}{b} + b'f',
\end{align} 
where $(f, a, b): \R_{\geq 0} \rightarrow \R^3$ are functions depending on $s$ only and $n$ is the complex dimension of the base manifold $\hat{M}$. Note that we obtain the same soliton equations for metrics (\ref{warped}) and (\ref{warped_hd}), because they are special cases of the general metric (\ref{metric}).

The boundary conditions on $a$, $b$ and $f$, which ensure smoothness of the metric at $s=0$, depend on the topology of the underlying manifold through $p$ and $k$. In particular, the period of $\tau$ is equal to $\Delta \tau =\frac{2\pi p}{(n+1)k}$, which follows by either considering the holonomy of the connection $A$ or the construction of the line bundle given the Chern class $k \sigma$. Therefore we must require $a'(0) \Delta \tau = 2\pi$ in order to avoid a conical singularity at $s=0$. A sufficient condition to ensure smoothness of the metric and potential function $f$ at $s=0$ is that $a$ is smoothly extendable to an odd function and $b,f$ are smoothly extendable to even functions around $s=0$. Notice that equations $(\ref{solitoneqn1})-(\ref{solitoneqn3})$ depend on $f'$ and $f''$ only, allowing us to assume without loss of generality $f(0) = 0$. Finally, by the scaling symmetry $g \rightarrow \alpha g$, $\alpha \in \R$, we can fix $b(0) = 1$. In summary our boundary conditions at $s=0$ read
\begin{align}
\label{boundaryconditions}
a = 0 && a' &= (n+1)\frac{k}{p} \\ \nonumber
b = 1 && b' &= 0 \\ \nonumber
f = 0 && f' &= 0.
\end{align}
For sections 2-7 we assume these boundary conditions to hold in all lemmas and theorems stated. Note when the underlying manifold is the complex line bundle $O(-k)$ over $\mathbb{C}P^n$, we have $p = n+1$ and therefore $a'(0) = k$.

Equations $(\ref{solitoneqn1})-(\ref{solitoneqn3})$ with the above boundary conditions (\ref{boundaryconditions}) are degenerate at $s=0$ and we must specify $f''(0)$ to obtain a unique solution. This is explained in Appendix B, where we prove the following theorem:

\begin{thm}
\label{analiticity}
Fix $n \in \N$ and $a_0, f_0^{\ast} \in \R$. Then there exists an $\epsilon > 0$ such that
\begin{enumerate}
\item For any $|f_0-f_0^{\ast}|<\epsilon$ there exists a unique analytic solution $(f,a,b): (-\epsilon, \epsilon) \setminus \{0\}\rightarrow \R^3$ to the soliton equations (\ref{solitoneqn1})-(\ref{solitoneqn3}) satisfying the initial conditions $a(0) = 0$, $a'(0) = a_0$, $b(0) = 1$, $b'(0) = 0$, $f(0) = f'(0) = 0$ and $f''(0) = f_0$.
\item $a$ is an odd function and $b,f$ are even functions
\item The solution $(f,a,b)$ depends analytically on $f_0$
\end{enumerate}
\end{thm}

This theorem, in conjunction with standard theory of ordinary differential equations, shows that solutions to the soliton equations $(\ref{solitoneqn1})-(\ref{solitoneqn3})$ depend smoothly on $f''(0)$.

\section{Evolution equations for $Q= \frac{a}{b}$, $f$ and $R$}
\label{Qevol}
A central quantity in this paper is the quotient $Q:= \frac{a}{b}$. From the soliton equations (\ref{solitoneqn1})-(\ref{solitoneqn3}) we can compute its evolution equation
\begin{equation}
\label{Qeqn}
Q'' = \left( f' - (2n+1) \frac{b'}{b}\right) Q' + \frac{2n+2}{b^2}\left(Q^3- Q \right),
\end{equation}
which leads to the following key lemma:
\begin{lem}
\label{QkeyLemma}
Let $(f,a,b): I \rightarrow \R^3$ be a solution to the soliton equations (\ref{solitoneqn1})-(\ref{solitoneqn3}). Assume that $s \in I$ is a local extrema of $Q$. If $Q(s) > 1$ $(0 < Q(s) < 1)$ then $Q$ has a strict local minimum (maximum) at $s$. Moreover
\end{lem}

\begin{proof}
The proof follows from the evolution equation (\ref{Qeqn}) of $Q$.
\end{proof}
\begin{remark}
From Lemma \ref{QkeyLemma} and the boundary condition $Q(0) = 0$ it follows that if for some $s_0 > 0$ we have $Q(s_0)>1$ then $Q'(s)>0$ for all $s>s_0$.
\end{remark}
Via the Bianchi identity we obtain the following lemma:
\begin{lem}
\label{bianchi-Lemma}
For a steady gradient soliton $R_{ij} + \nabla_i \nabla_j f = 0$ the identity
\begin{equation}
\label{bianchi}
\frac{1}{2} \nabla_j R = \nabla_k R^{\;k}_j = R_j^{\;k}\nabla_k f 
\end{equation}
holds true. A consequence is that $\nabla_i \left(R + |\nabla f|^2\right) = 0$ and therefore $R + |\nabla f|^2$ is constant.
\end{lem}
\begin{proof}
The fist equality of (\ref{bianchi}) is just the contracted Bianchi identity. Computing
\begin{align*}
\nabla_k R^k_{\;j} &= - \nabla_k \nabla^k\nabla_j f \\
&=(\nabla_j \nabla_k -\nabla_k \nabla_j) \nabla^k f - \nabla_j \nabla_k \nabla^k f \\
&=R^k_{\; ajk} \nabla^a f + \nabla_j R \\
&=- R_{aj} \nabla^a f + 2\nabla_k R^k_{\;j}
\end{align*}
we obtain the second equality after rearranging terms. A simple computation then shows that $\nabla_i \left(R + |\nabla f|^2\right) = 0$.
 \end{proof}
From the lemma above we derive evolution equations for $f$ and $R$.

\begin{lem}
\label{lem:feqn-integrated}
The potential function $f$ of a steady gradient soliton satisfies
\begin{equation}
\label{feqn}
\Delta f -|\nabla f|^2  = - R(0) = 2f''(0)
\end{equation}
\end{lem}
\begin{proof}
The first equality follows from Lemma \ref{bianchi-Lemma}, the fact that $R + \Delta f = 0$, and the boundary condition $f'(0) = 0$.

To prove the second, note that the second derivatives in the expression (\ref{scalarcurvature}) for $R$ can be eliminated with help of the soliton equations (\ref{solitoneqn2})-(\ref{solitoneqn3}). Carrying this out you obtain 
\begin{equation}
\label{Rfirstorder}
R = 2n \frac{a^2}{b^4} - \frac{2n(2n+2)}{b^2} + 4n \frac{a'b'}{ab}+ 2n(2n-1)\left(\frac{b'}{b}\right)^2 -2f'\left(\frac{a'}{a} + 2n \frac{b'}{b}\right).
\end{equation} 
Now apply L'H\^opital's rule around $s=0$ and note $b''(0) = n+1$ by (\ref{solitoneqn3}) to deduce that $R(0) = -2f''(0)$.
\end{proof}

This leads to the following corollary:

\begin{cor}
\label{cor:f-bounds}
A complete solution to the soliton equations (\ref{solitoneqn2})-(\ref{solitoneqn3}) satisfies (i) $f''(0) \leq 0$ and (ii) $(f')^2 \leq R(0) = -2 f''(0)$.
\end{cor}
\begin{proof}
It is a well-known fact that $R\geq0$ for any complete ancient solution to Ricci flow (see for instance \cite[Corollary 2.5]{Chen09}). Hence (i) follows by Lemma \ref{lem:feqn-integrated} above. For (ii) note that $R + |\nabla f|^2 = - 2 f''(0)$ and $|\nabla f|^2 = (f')^2$.
\end{proof}

We assume $f''(0) \leq 0$ for the remainder of the paper. Below we derive an evolution equation for $R$.
\begin{lem} The scalar curvature of a gradient steady soliton satisfies
 \begin{equation}
 \label{Reqn}
\Delta R + 2 |Ric|^2 = \nabla^i R \nabla_i f
 \end{equation}
\end{lem}
\begin{proof}
Applying the Bianchi identity (\ref{bianchi}) we obtain
 \begin{equation*}
 \nabla^j\nabla_j R = 2 \left( \nabla^j R_j^{\;k}\right)\nabla_k f + 2 R_j^{\;k} \nabla_k \nabla^j f,
 \end{equation*}
from which the desired result follows.
\end{proof}

\section{Monotonicity properties of $a$, $b$, $f$, $f'$ and $R$}
Using the soliton equations (\ref{solitoneqn1})-(\ref{solitoneqn3}) and evolution equations for $f$ and $R$ derived in the section above, we deduce various monotonicity properties of $a$, $b$, $f$, $f'$ and $R$ for $Q< \sqrt{n+1}$.

\begin{lem}
\label{ab_mono}
Let $s_0 >0$ and $(f,a,b): [0,s_0) \rightarrow \R^3$ be a solution to the soliton equations (\ref{solitoneqn1})-(\ref{solitoneqn3}). Then $a$ is strictly increasing on $[0,s_0)$ and $b$ is strictly increasing on any interval $0 \in I \subset [0,s_0)$ on which $Q<\sqrt{n+1}$. Moreover, $b'$ changes its sign at most once in the interval $[0,s_0)$.
\end{lem}
\begin{proof}
The evolution equation (\ref{solitoneqn2}) of $a$ implies that $a'' = 2n \frac{a^3}{b^4}$ whenever $a' = 0$. By the boundary conditions (\ref{boundaryconditions}) we have $a'(0) > 0$, and therefore it follows that $a$ is strictly increasing. Similarly, we deduce from the evolution equation (\ref{solitoneqn3}) of $b$ that $b'' = 2\frac{n+1 - Q^2}{b}$ whenever $b'=0$. Applying L'H\^opital's rule around $s=0$ shows that $b''(0) = n+1>0$. This in conjunction with the boundary condition $b'(0)=0$ implies that $b$ is strictly increasing on any interval $I = [0,s], s>0,$ where $Q < \sqrt{n+1}$. Therefore $b'$ can change its sign only when $Q^2\geq n+1$. Since $Q$ is strictly increasing when $Q>1$ and $b'' = 2\frac{n+1 - Q^2}{b}$ whenever $b'=0$, it follows that $b'$ changes its sign at most once.
\end{proof}
We can prove the monotonicity properties of $f$ in a similar fashion:
\begin{lem}
\label{f_mono}
Let $s_0>0$ and $(f,a,b): [0,s_0) \rightarrow \R^3$ be a solution to the soliton equations (\ref{solitoneqn1})-(\ref{solitoneqn3}). Then,
\begin{enumerate}
\item if $f''(0)<0$ then $f$ and $f'$ are strictly decreasing functions, in particular
	\begin{enumerate}
	\item $f' < 0$ for $s >0$
	\item $f'' < 0$ for $s \geq 0$
	\end{enumerate}
\item if $f''(0) = 0$ then $f\equiv0$
\end{enumerate}
\end{lem}
\begin{proof}
Using the expression (\ref{laplacian}), we may write equation (\ref{feqn}) for $f$ as
\begin{equation}
\label{flocalcoordinates}
f'' + \left(\frac{a'}{a}+2n\frac{b'}{b}\right)f' - (f')^2 = 2f''(0)
\end{equation}
Hence at an extremal point of $f$ we have $f'' = 2f''(0) <0$. This in conjunction with the boundary condition $f'(0)=0$ proves that $f$ is strictly decreasing and $f' < 0$ for $s >0$. 

By the soliton equations (\ref{solitoneqn1}) - (\ref{solitoneqn3}) we have
\begin{equation*}
\left( \frac{a'}{a} + 2n \frac{b'}{b}\right)' = f'' - \left ( \left(\frac{a'}{a}\right)^2 +  2n\left(\frac{b'}{b}\right)^2\right ).
\end{equation*}
Differentiating (\ref{flocalcoordinates}) we therefore obtain
\begin{equation*}
f''' = f'f'' + \left ( \left(\frac{a'}{a}\right)^2 +  2n\left(\frac{b'}{b}\right)^2 \right) f' - \left( \frac{a'}{a} + 2n \frac{b'}{b}\right) f''.
\end{equation*}
Hence
\begin{equation}
\label{f3deriv-extrema}
f''' =\left ( \left(\frac{a'}{a}\right)^2 +  2n\left(\frac{b'}{b}\right)^2\right) f' < 0
\end{equation}
whenever $f''=0$. The strict inequality follows from noting that $f''=0$ can only hold true for some $s>0$, because $f''(0) < 0$ by assumption, and by Lemma \ref{ab_mono} we have $a' >0$ whenever $s>0$. This proves that $f'' < 0$. 

We now prove statement (2). The continuous dependence on $f''(0)$ of solutions $(f,a,b)$ to the soliton equations (\ref{solitoneqn1})-(\ref{solitoneqn3}) and statement (1) imply that $f \leq 0$ everywhere. As $a'(0)>0$ the expression $\frac{a'}{a} + 2n\frac{b'}{b}$ is Lipshitz continuous on any closed interval $I\subset (0,s_0)$. Applying standard theory of ordinary differential equations, it follows from equation (\ref{flocalcoordinates}) that if $f$ is constantly zero in a neighborhood of $s=0$ it must be constantly zero on all of $[0,s_0)$. Therefore it suffices to show that $f$ is zero near $s=0$. Assume this is not the case. Then there exists an interval of the form $(0,\epsilon)$, $\epsilon>0$, on which $f,f'<0$. The boundary conditions (\ref{boundaryconditions}) imply, after restricting to a smaller $\epsilon>0$ if necessary, that $\frac{a'}{a} +  2n\frac{b'}{b}>0$ on $(0,\epsilon)$. This, however, leads to a contradiction as (\ref{flocalcoordinates}) then implies that $f''>0$ on $(0,\epsilon)$. 
\end{proof}

We obtain the following corollary from the above lemma:

\begin{cor}
\label{df-lower-bound}
Let $s_0>0$ and $(f,a,b): [0,s_0) \rightarrow \R^3$ be a solution to the soliton equations (\ref{solitoneqn1})-(\ref{solitoneqn3}) with $f''(0) < 0$ and
$Q < \sqrt{n+1}$. Then $f' \geq - \sqrt{-2f''(0)}$.
\end{cor}
\begin{proof}
By Lemma \ref{ab_mono} both $a$ and $b$ are strictly increasing on $[0,s_0)$. Hence by Lemma \ref{f_mono} and equation (\ref{flocalcoordinates}) the result follows.
\end{proof}

Finally, we also prove that $R$ is monotonically decreasing.
\begin{lem}
\label{R_mono}
Let $s_0>0 $ and $(f,a,b): [0, s_0) \rightarrow \R^3$ be a solution to the soliton equations (\ref{solitoneqn1})-(\ref{solitoneqn3}) with $f''(0) <0$. Then $R$ is a strictly decreasing function of $s$.
\end{lem}
\begin{proof}
The evolution equation (\ref{Reqn}) for $R$ and the previous lemma imply that
$$\Delta R = R'' \leq -2|Ric|^2 = -2|\nabla_i\nabla_j f|^2 < 0,$$
whenever $R' = 0$. Since $R'= 0$ and $R''<0$ at $s=0$ we obtain the desired result.
\end{proof}

\section{Existence of complete solitons}
In this section we will prove the following theorem:
\begin{thm}
\label{thm-collapsed}
On the line bundle $L_{k,p}$ for $k, p \in \N$ there exists a family of complete steady Ricci solitons. In particular, there exists a $f_0 \geq 0$ such that any $f''(0) < -f_0$ yields a complete solution $(f,a,b): [0,\infty) \rightarrow \R^3$ to the soliton equations (\ref{solitoneqn1})-(\ref{solitoneqn3}).
\end{thm}

\begin{remark}
These solitons were constructed independently in \cite{Wink17} and \cite{Stol17}.
\end{remark}

The strategy is to show that as long as $Q  < \sqrt{n+1}$ a solution cannot blow up in finite distance and then use the evolution equation (\ref{Qeqn}) of $Q$ to argue that one can keep $Q$ arbitrarily small by picking $f''(0) \ll -1$.

\begin{lem}
\label{Qcontrol}
Let $s_0>0$ and $(f,a,b): [0, s_0) \rightarrow \R^3$ be a solution to the soliton equations (\ref{solitoneqn1})-(\ref{solitoneqn3}) with $f''(0) \leq 0$. If $Q < \sqrt{n+1}$ on $[0, s_0)$, the solution can be extended past $s_0$.
\end{lem}

\begin{proof}
The monotonicity properties of $a$, $b$ and $f$ derived in Lemma \ref{ab_mono} imply that whenever $Q< \sqrt{n+1}$
\begin{align*}
a'' &\leq 2n\frac{a^3}{b^4} \leq \frac{2n(n+1)^{\frac{3}{2}}}{b} \leq  2n(n+1)^{\frac{3}{2}}\\
b'' &\leq\frac{2n+2}{b} - 2\frac{a^2}{b^3}\leq 2 \frac{n+1-Q^2}{b} \leq 2 (n+1),
\end{align*}
which in turn shows that
\begin{align*}
a'(s) &< a'(0) + 2n(n+1)^{\frac{3}{2}}s && a(s) <  a'(0) s +  n(n+1)^{\frac{3}{2}} s^2 \\
b'(s) &< 2(n+1)s && b(s) < 1 + (n+1) s^2 
\end{align*}
as long as $Q < \sqrt{n+1}$ holds true. Moreover, since $a', b' > 0$ for all $s \in (0,s_0)$, there exists a $c>0$ such that $a, b > c$ for $s \in [\frac{s_0}{2}, s_0)$. Finally, recall that by Corollary \ref{df-lower-bound} we have
\begin{equation*}
 -\sqrt{-2f''(0)} \leq f' \leq 0.
\end{equation*}
Applying the Picard–Lindelöf theorem we conclude that the solution may be extended past $s_0$.

\end{proof}
Now we show that for at least short distance $s$ we have $Q < \sqrt{n+1}$.

\begin{lem}
\label{Qbound}
Let $(f,a,b): [0, s_0) \rightarrow \R^3$ be a solution to the soliton equations (\ref{solitoneqn1})-(\ref{solitoneqn3}). Then $Q(s) \leq a'(0)s$ for $s \leq \frac{1}{a'(0)}$. Thus the solution may be extended to $[0,\frac{\sqrt{n+1}}{a'(0)})$.
\end{lem}
\begin{proof}
From the evolution equation (\ref{Qeqn}) for $Q$ we have that whenever $0\leq Q \leq 1$ and $Q'>0$ 
\begin{equation*}
\left[\ln Q'\right]' \leq \left[f -(2n+1) \ln b \right]'.
\end{equation*}
Integrating we obtain
\begin{equation}
\label{Qineq}
Q'(s)\leq Q'(0) \frac{e^{f(s)}}{b(s)^{2n+1}} \leq a'(0)
\end{equation}
by the monotonicity properties of $f$ and $b$, and the fact that $Q'(0) = a'(0)$. Integrating again, yields the desired result by Lemma \ref{Qcontrol}.
\end{proof}

We now prove Theorem \ref{thm-collapsed}:

\begin{proof}[Proof of Theorem \ref{thm-collapsed}]
From the soliton equations (\ref{solitoneqn1})-(\ref{solitoneqn3}) it follows that
\begin{align*}
f'' &= \frac{a''}{a} + 2n \frac{b''}{b} \\ 
	&= - 2n \frac{a^2}{b^4}+ \frac{4n(n+1)}{b^2}- 4n \frac{a'b'}{ab} - 2n(2n-1)\left(\frac{b'}{b}\right)^2 + \left(\frac{a'}{a}+ 2n \frac{b'}{b}\right) f'.
\end{align*}
Solving this equation for $\left(\frac{a'}{a}+ 2n \frac{b'}{b}\right) f'$ and substituting the resulting expression into the evolution equation (\ref{flocalcoordinates}) of $f$ yields
\begin{equation*}
 \label{eqn2} f'' = f''(0) - n \frac{a^2}{b^4}+ \frac{2n(n+1)}{b^2}- 2n \frac{a'b'}{ab}- n(2n-1)\left(\frac{b'}{b}\right)^2 + \frac{(f')^2}{2}
\end{equation*}
As long as $a$ and $b$ are increasing, which by Lemma \ref{Qbound} is true for $s <\frac{1}{a'(0)}$, we have
\begin{equation}
\label{f-diff-ineq}
f'' <  f''(0) + 2n(n+1) + \frac{(f')^2}{2}
\end{equation}
Notice that for any $s_0 \in (0, \frac{1}{a'(0)})$ and $c_0 >0$ there exists a $f_0>0$ such that for any function $f$ satisfying the differential inequality (\ref{f-diff-ineq}) above, subject to the boundary conditions $f'(0)= 0$ and $f''(0) < -f_0$, we have $f'(s_0) \leq -c_0$. The monotonicity properties of $f$ derived in Lemma \ref{f_mono} then imply that $f(s) \leq -c_0(s-s_0)$ for $s>s_0$.

As long as $Q \leq 1$ and $Q' \geq 0$, the inequality (\ref{Qineq}) holds and $b \geq 1$ by Lemma \ref{ab_mono}. Hence integrating inequality (\ref{Qineq}) implies that

\begin{align*}
Q &\leq a'(0) \left(s_0 + \int_{s_0}^{\infty} e^{-c_0(s-s_0)} \right)\\
  &\leq a'(0)\left(s_0 + \frac{1}{c_0} \right),
\end{align*}

as long as $Q \leq 1$ and $Q' \geq 0$. Thus, choosing $s_0$ small and $c_0$ large we can ensure that $Q < 1$ for all $s \geq 0$ or until we reach a point $s^\ast >0$ at which $Q'(s^\ast) = 0$. In the latter case, however, Lemma \ref{QkeyLemma} implies that $Q' \leq 0 $ for all $s > s^\ast$, and therefore $Q < 1$ everywhere. Now Lemma \ref{Qcontrol} yields the desired result.
\end{proof}

\begin{remark}
\label{rmk-collapsed}
From the above proof it follows that for any $c>0$ there exists a $f_0 >0$ such that a solution $(f,a,b): [0, \infty) \rightarrow \R^3$ to the soliton equations (\ref{solitoneqn1})-(\ref{solitoneqn3}) with $f''(0) < -f_0$ satisfies $Q \leq c$ everywhere.
\end{remark}

\section{Asymptotics}
In this section we study the behavior of $Q$ as $s\rightarrow \infty$ in the case that $f''(0) < 0$. The goal is to prove the following theorem:
\begin{thm}
\label{thm-asymptotics}
Let $(f,a,b): [0,\infty) \rightarrow \R^3$ be a solution to the soliton equations (\ref{solitoneqn1})-(\ref{solitoneqn3}) with $f''(0)<0$. Then either $\lim_{s\rightarrow\infty} Q = 0$ or $\lim_{s\rightarrow\infty} Q = 1$. Furthermore 
\begin{enumerate}
\item if $\lim_{s\rightarrow \infty} Q = 0$ we have $a \sim const$ and $b\sim const \sqrt{s}$
\item if $\lim_{s\rightarrow \infty} Q = 1$ we have $a \sim b\sim const \sqrt{s}$
\end{enumerate}
as $s \rightarrow \infty$. Finally, any complete solution to the soliton equations satisfies $Q\leq 1$ everywhere.
\end{thm}
It is useful to rewrite the soliton equations (\ref{solitoneqn2}) and (\ref{solitoneqn3}) for $a$ and $b$ in the form
\begin{align}
\label{aQ} a'' &= \frac{2nQ^4}{a} - 2n \frac{(a')^2}{a} + \left(f' +2n \frac{Q'}{Q} \right) a' \\
\label{bQ} b'' &= 2\frac{n+1-Q^2}{b} - 2n \frac{(b')^2}{b}  + \left(f'- \frac{Q'}{Q}\right) b'.
\end{align}
As the limits of both $f'$ and $Q$ as $s\rightarrow \infty$ exists, we will be able to derive the asymptotics of $a$ and $b$ from the following auxiliary lemma:

\begin{lem} 
\label{asymptotics}
Let $\alpha>1$, $\epsilon >0$ and $c_1^{\ast}, c_2^{\ast} > \epsilon$. Assume $c_i: [0, \infty) \rightarrow \R$, $i = 1,2$, are two positive smooth functions satisfying
\begin{equation*}
|c_i(s) - c_i^{\ast}| < \epsilon, \: i = 1,2,
\end{equation*}
for all $s\geq 0$. Then for a solution $y: [0, \infty) \rightarrow \R$ to the ODE
\begin{equation}
\label{ODEasymptotics}
y'' = \frac{c_1(s)}{2y} - \alpha \frac{(y')^2}{y} - c_2(s) y'
\end{equation}
with initial conditions $y(0), y'(0)>0$ there exists an $s_0>0$ such that for $s >s_0$ 
\begin{equation*}
y^2(s_0) + \gamma_-\left(1+\epsilon\right)^{-1} \left(s-s_0\right) \leq y^2(s) \leq y^2(s_0) + \gamma_+ \left(s-s_0\right),
\end{equation*}
where 
\begin{equation*}
\gamma_{\pm} = \frac{c_1^{\ast} \pm \epsilon}{c_2^{\ast}\mp\epsilon}.
\end{equation*}
\end{lem}
\begin{proof}
Substituting $z = y^{\alpha+1}$ we can rewrite ODE (\ref{ODEasymptotics}) as
\begin{equation*}
z'' = \frac{\alpha+1}{2} c_1(s) z^{\frac{\alpha-1}{\alpha+1}} - c_2(s)z'.
\end{equation*}
Define $w = z'$ and $f(z,s) = \frac{\alpha+1}{2} \frac{c_1(s)}{c_2(s)} z^{\frac{\alpha-1}{\alpha+1}}$, yielding the system of equations
\begin{align*}
z' &= w \\
w' &= c_2(s) \left(f(z,s)-w\right).  
\end{align*}
We investigate the phase diagram of this ODE system in the first quadrant $w>0 , z>0$ (where we take $z$ to be the $x$-axis and $w$ to be the $y$-axis).  Consider the subregions 
\begin{align*}
R_-&: \qquad 0 < w < f_-(z) (1 + \epsilon)^{-1} \\
R_+&:\qquad    w > f_+(z) \\
S&:\qquad  f_-(z) (1 + \epsilon)^{-1} < w < f_+(z)
\end{align*}
of the first quadrant, where
\begin{align*}
f_{\pm}(z) & = \frac{(\alpha+1)}{2} \frac{c_1^{\ast} \pm \epsilon}{c_2^{\ast}\mp\epsilon} z^{\frac{\alpha-1}{\alpha+1}} \equiv \frac{(\alpha+1)}{2} \gamma_{\pm} z^{\frac{\alpha-1}{\alpha+1}}.
\end{align*}
Note that picking $\epsilon>0$ sufficiently small we have $0<f_-(z) < f(z,s) < f_+(z)$. In the subregion $R_-$
\begin{equation*}
\frac{\mathrm{d}w}{\mathrm{d}z} = c_2(s) \left( \frac{f(z,s)}{w} -1 \right)> \left(c_2^{\ast} - \epsilon\right) \epsilon
\end{equation*}
and in the subregion $R_+$
\begin{equation*}
\frac{\mathrm{d}w}{\mathrm{d}z} = c_2(s) \left( \frac{f(z,s)}{w} -1 \right)<0.
\end{equation*}
Because $f_+(z)$ is strictly increasing, any solution starting in $R_+$ will eventually enter $S$ and never return to $R_+$. Similarly $\lim_{ z \rightarrow \infty} f'_-(z) = 0$ implies that any solution starting in $R_-$ will eventually leave $R_-$. We conclude that there exists an $s_0>0$ such that the points $(w(s), z(s)), s > s_0,$ lie in the subregion $S$ of the first quadrant. Thus for $s>s_0$
\begin{equation}
\label{eqn:asymptotics-derivative}
\frac{z'}{z^{\frac{\alpha-1}{\alpha+1}}}  = \frac{\alpha+1}{2}\gamma(s)\left(1 + \epsilon(s)\right)^{-1},
\end{equation}
where $\gamma(s)$ and $\epsilon(s)$ are functions in the range $(\gamma_- , \gamma_+)$ and $(0,\epsilon)$ respectively. Integrate this equation from $s_0$ to $s$ and re-substitute $y$ to obtain the desired result.
\end{proof}

We can prove a slight generalization to Lemma \ref{asymptotics}, by considering the case in which $0< c_1(s) \leq \epsilon$:

\begin{cor}
\label{asymptotics_cor}
In the same setup as in Lemma \ref{asymptotics}, apart from the assumption that $c_1^{\ast} = 0$, it follows that there exists an $s_0 >0$ such that for $s > s_0$
\begin{equation*}
y^2(s_0) \leq y^2(s) \leq y^2(s_0) + \frac{\epsilon}{c_2^{\ast}-\epsilon}\left(s-s_0\right).
\end{equation*}
\end{cor}
\begin{proof}
Notice that $y$ must be non-decreasing, which proves the lower bound. For the upper bound, we follow the proof of Lemma \ref{asymptotics}. This time, however, we only consider the subregions 
\begin{align*}
R_+:&\qquad    w > f_+(z) \\ 
R_-:&\qquad    w \leq f_+(z)
\end{align*}
of the first quadrant of the phase diagram, where 
\begin{equation*}
f_+(z) = \frac{\alpha+1}{2} \frac{\epsilon}{c_2^{\ast} - \epsilon} z ^ {\frac{\alpha-1}{\alpha+1}}.
\end{equation*}
Any solution starting from $R_+$ will eventually enter $R_-$ and remain there. Moreover, any solution starting from $R_-$ remains in $R_-$. Hence there exists an $s_0 > 0$ such that for $s > s_0$
\begin{equation*}
w \leq f_+(z).
\end{equation*}
Integrating this equation gives the desired result.
\end{proof}

Now we proceed to prove:
\begin{lem}
\label{criticalLemma}
Let $(f,a,b): [0,\infty) \rightarrow \R^3$ be a complete solution to the soliton equations (\ref{solitoneqn1})-(\ref{solitoneqn3}). 
Then the limit $Q_{\infty} :=\lim_{s\rightarrow \infty} Q$ exists. Moreover, if $Q_{\infty} < \infty$ and $f''(0) < 0$ then $\lim_{s\rightarrow\infty} Q' = 0$.
\end{lem}
\begin{proof}
By Lemma \ref{QkeyLemma} and Lemma \ref{ab_mono} we know that both $Q'$ and $b'$ change their sign at most once. Therefore the limits 
\begin{align*}
Q_{\infty}&:=\lim_{s\rightarrow \infty} Q   \\
b_{\infty}&:= \lim_{s\rightarrow \infty} b
\end{align*}
both exist. As $f''(0)<0$ by assumption if follows from Lemma \ref{f_mono} that the limit
\begin{equation*}
f'_{\infty} := \lim_{s\rightarrow\infty} f'(s)  < 0 
\end{equation*}
exists and is negative. By Corollary \ref{cor:f-bounds} we also know that $f'_\infty > -\infty$. Rewrite the evolution equation (\ref{Qeqn}) for $Q$ as
\begin{equation*}
\left[b^{2n+1}e^{-f}Q'\right]' = (2n+2)b^{2n-1}e^{-f}\left(Q^3-Q\right)
\end{equation*}
and integrate from $0$ to $s$. This yields 
\begin{equation}
\label{dQ_eqn}
Q'(s) = \frac{Q'(0)e^{f(s)}}{b(s)^{2n+1}} + (2n+2)\frac{e^{f(s)}}{b(s)^{2n+1}}\int_{0}^s b(t)^{2n-1}e^{-f(t)}\left[Q(t)^3-Q(t)\right]\,\mathrm{d}t.
\end{equation}
We distinguish the following three cases:

\vspace{0.5em}
\noindent\textbf{Case 1: $b_{\infty} = 0$}
\vspace{0.5em}

\noindent Lemma \ref{ab_mono} implies that there exists an $s_0 >0$ such that for $s>s_0$ we have $Q^2 > n + 1$ and $b' < 0$. Therefore it follows that
\begin{align*}
\lim_{s \rightarrow \infty} \frac{e^{f}}{b^{2n-1}}\int_{s_0}^{s} b^{2n-1}e^{-f}\left(Q^3-Q\right) \; \mathrm{d}s &\geq n\sqrt{n+1} \lim_{s \rightarrow \infty} e^{f} \int_{s_0}^s e^{-f}\; \mathrm{d}s \\
&= - \frac{n\sqrt{n+1}}{f'_{\infty}} > 0,
\end{align*}
implying that the right hand side of (\ref{dQ_eqn}) tends to $\infty$ as $s \rightarrow \infty$, contradicting our assumption that $Q_{\infty} < \infty$. Hence $b_\infty = 0$ cannot occur. 

\vspace{0.5em}
\noindent\textbf{Case 2: $0 <b_{\infty} < \infty$}
\vspace{0.5em}

\noindent Apply L'H\^opital's rule to the right hand side of (\ref{dQ_eqn}) to compute the limit
\begin{equation*}
\label{Q_lim}
\lim_{s\rightarrow\infty} Q' = -(2n+2)\frac{1}{b_{\infty}^2 f'_{\infty}}\left(Q_{\infty}^3- Q_{\infty}\right).
\end{equation*}
As $0\leq Q_{\infty}< \infty$ it follows that $\lim_{s\rightarrow\infty} Q' = 0$. Moreover, $Q_{\infty} = 0$ or $1$ in this case.

\vspace{0.5em}
\noindent\textbf{Case 3: $b_{\infty} = \infty$}
\vspace{0.5em}

\noindent By Lemma \ref{ab_mono} $b$ is increasing and by Lemma \ref{f_mono} both $f$ and $f'$ are strictly decreasing. Hence there exists a $c>0$ such that $f(s) < - cs$ for sufficiently large $s>0$. Using (\ref{dQ_eqn}), this yields the bound
$$ |Q'(s)| \leq C e^{-cs} + C \frac{e^{f(s)}}{b(s)^2} \int_{0}^s e^{-f(t)} \left| Q^3(t)- Q(t) \right| \: \mathrm{d}t $$
for some $C>0$. By L'H\^opital's Rule  
$$ \lim_{s\rightarrow \infty} e^{f(s)} \int_{0}^s e^{-f(t)} \left| Q^3(t)- Q(t) \right| \: \mathrm{d}t = \frac{\left| Q_\infty^3 - Q_\infty \right|}{-f'_{\infty}} < \infty,$$
and therefore $\lim_{s\rightarrow \infty} Q' =0$. This concludes the proof.

\end{proof}

For the remainder of the paper we write $Q_{\infty} := \lim_{s\rightarrow\infty} Q$ and $f_{\infty} = \lim_{s\rightarrow\infty} f$.
\begin{lem}
\label{Qinftyn+1}
There are no complete solutions $(f,a,b): [0,\infty) \rightarrow \R^3$ to the soliton equations (\ref{solitoneqn1})-(\ref{solitoneqn3}) with $f''(0) < 0$ and $n + 1 < Q^2_{\infty}<\infty$.
\end{lem}
\begin{proof}
Assume such a solution exists. By the proof of Lemma \ref{criticalLemma} above, we then know that $b_{\infty}= \infty$, which in turn implies that $b$ is monotonically increasing by Lemma \ref{ab_mono}. Furthermore $Q' \geq 0$ and $f' < 0$ by Lemma \ref{QkeyLemma} and Lemma \ref{f_mono}. Choose $s_0, c>0$ such that $Q^2 \geq n + 1 + c$ for $s_0>s$. From (\ref{bQ}) it then follows that
\begin{equation*}
b'' < - \frac{2c}{b} \quad \text{for} \quad s > s_0.
\end{equation*}
Multiplying this inequality by $b'$ and integrating from $s_0$ to $s$, we deduce
\begin{equation*}
b'(s)^2 \leq b'(s_0)^2 - 4c \ln \frac{b(s)}{b(s_0)}.
\end{equation*}
Therefore $b'$ must become negative in finite distance $s$, contradicting the monotonicity of $b$.
\end{proof}

\begin{lem}
\label{Qlim}
Let $(f,a,b): [0,\infty) \rightarrow \R^3$ be a complete solution to the soliton equations (\ref{solitoneqn1})-(\ref{solitoneqn3}) with $Q_{\infty}<\infty$ and $f''(0) < 0$. Then either (i) $\lim_{s \rightarrow \infty} Q = 0$ or (ii) $\lim_{s \rightarrow \infty} Q = 1$.
\end{lem}
\begin{proof}
We show that $Q_{\infty} \neq 0$ implies that $Q_{\infty} = 1$. Recall that
\begin{itemize}
\item $Q^2_{\infty} \leq n+1$ by Lemma \ref{Qinftyn+1}
\item $\lim_{s \rightarrow \infty} Q' = 0$ by Lemma \ref{criticalLemma}
\item $\lim_{s \rightarrow \infty} f' = f'_{\infty} < 0$ by Lemma \ref{f_mono}
\end{itemize}

We first prove the result for $0 < Q^2_{\infty} < n+1$. Apply Lemma \ref{asymptotics} to (\ref{aQ}) and (\ref{bQ}) to deduce that for sufficiently small $\epsilon >0$ there exists an $s_0> 0$ such that for $s>s_0$
\begin{equation}
\label{eqn1}
\frac{a(s_0)^2 + \gamma_{a,-}\left(1+\epsilon\right)^{-1} \left(s-s_0\right)}{b(s_0)^2 + \gamma_{b,+} \left(s-s_0\right)}
\leq \frac{a^2(s)}{b^2(s)} \leq  \frac{a(s_0)^2 + \gamma_{a,+} \left(s-s_0\right)}{b(s_0)^2 + \gamma_{b,-}\left(1+\epsilon\right)^{-1} \left(s-s_0\right)},  
\end{equation}
where
\begin{align*}
\gamma_{a, \pm} &= \frac{4nQ_{\infty}^4 \pm \epsilon} {-f'_{\infty}\mp \epsilon} \\
\gamma_{b, \pm} &= \frac{4(n+1-Q_{\infty}^2) \pm \epsilon} {-f'_{\infty}\mp \epsilon}.
\end{align*}
Taking the limit of (\ref{eqn1}) as $s \rightarrow \infty$ we obtain 
\begin{equation*}
\frac{\gamma_{a,-}}{\gamma_{b,+}} \leq Q_{\infty}^2 \leq \frac{\gamma_{a,+}}{\gamma_{b,-}}. 
\end{equation*}
Since $\epsilon >0$ can be chosen arbitrarily small, we conclude that $Q_{\infty}$ solves the equation 
\begin{equation*}
\frac{nQ_{\infty}^4}{n+1-Q_{\infty}^2} = Q^2_{\infty},
\end{equation*}
which in the interval $(0,\sqrt{n+1})$ has the unique solution $Q_{\infty}=1$.

It remains to prove the result for $Q^2_{\infty} = n + 1$. Recall that by Lemma \ref{criticalLemma} we have $Q' > 0$ and hence $Q < n + 1$ everywhere in this case. Thus it follows from Corollary \ref{asymptotics_cor} that
\begin{equation*}
Q^2(s) = \frac{a^2(s)}{b^2(s)} \geq \frac{a^2(s_0) + \gamma_{a,-}\left(1+\epsilon\right)^{-1}\left(s-s_0\right) }{b^2(s_0) + \gamma_{b,+}\left(s-s_0\right) }
\end{equation*}
This, however, is a contradiction of $Q^2_{\infty} = n+1$, since $\gamma_{a,-} = O(1)$, $\gamma_{b,+} = O(\epsilon)$, and $\epsilon$ may be chosen arbitrarily small.

\end{proof}

As a corollary of the proof of Lemma \ref{Qlim} above we obtain

\begin{cor}
\label{Qlim1-asymptotic}
Let $(f,a,b): [0,\infty) \rightarrow \R^3$ be a complete solution to the soliton equations (\ref{solitoneqn1})-(\ref{solitoneqn3}) with $Q_\infty = 1$ and $f''(0) < 0$. Then 
\begin{equation*}
a,b \sim \sqrt{\frac{4n}{-f'_\infty}} \sqrt{s} \quad \textrm{as} \quad s \rightarrow \infty.
\end{equation*}
\end{cor}

It remains to study the asymptotics of $a$ and $b$ when $Q_{\infty} = 0$. This is carried out in the following two Lemmas \ref{b-asymp-Q0} and \ref{a-asymp-Q0} below.

\begin{lem}
\label{b-asymp-Q0}
Let $(f,a,b): [0,\infty) \rightarrow \R^3$ be a complete solution to the soliton equations (\ref{solitoneqn1})-(\ref{solitoneqn3}) with $Q_\infty = 0$ and $f''(0) < 0$. Then 
\begin{align}
\label{b_asymp_Q0}
b &\sim b_0 \sqrt{s} \quad \textrm{as} \quad s \rightarrow \infty \\ \nonumber
b' &\sim \frac{b_0}{2\sqrt{s}} \quad \textrm{as} \quad s \rightarrow \infty
\end{align}
where 
\begin{equation*}
b_0 = \sqrt{ \frac{4(n+1)}{-f'_\infty}}
\end{equation*}
\end{lem}

\begin{proof}
Note that we cannot apply Lemma \ref{asymptotics} to equation (\ref{bQ}) to deduce the asymptotics of $b$ as $s \rightarrow \infty$, because a priori we do not know the limit of $\frac{Q'}{Q}$ as $s \rightarrow \infty$ when $Q_\infty = 0$. From the proof below it will follow that the limit is in fact $0$.

Let us first recall that the following properties hold:
\begin{itemize}
\item $Q < 1$ everywhere and for sufficiently large $s$ we have $Q'(s) < 0$. This follows from $\lim_{s \rightarrow \infty} Q = 0$ and Lemma \ref{QkeyLemma}.
\item $a', b' > 0$ for $s >0$. This follows from Lemma \ref{ab_mono}.
\item There exists a $c>0$ such that $f'(s) < -c$ for sufficiently large $s$. Moreover $\lim_{s \rightarrow \infty} f' = f'_\infty < 0$ exists. This follows from Lemma \ref{f_mono}.
\end{itemize}
Inspecting the evolution equation (\ref{aQ}) for $a$, the above properties show that for every $\epsilon >0$ there exists an $s^\ast>0$ such that for $s > s^\ast$
\begin{equation*}
a'' \leq \epsilon - c a'.
\end{equation*} 
This shows that $\lim_{s\rightarrow \infty} a' = 0$ and therefore $\lim_{s \rightarrow \infty} \frac{a'}{a} = 0$. Writing the evolution equation (\ref{solitoneqn3}) of $b$ as
\begin{equation*}
b'' = \frac{2n+2 - 2 Q^2}{b} - (2n-1) \frac{b'^2}{b} + \left(f' - \frac{a'}{a} \right)b'
\end{equation*}
and applying Lemma \ref{asymptotics} we deduce the asymptotics (\ref{b_asymp_Q0}) of $b$ as $s \rightarrow \infty$. From the proof of Lemma \ref{asymptotics}, 
in particular equation (\ref{eqn:asymptotics-derivative}), we also see that 
\begin{equation*}
b' \sim  \frac{b_0}{2 \sqrt{s}} \quad \textrm{as} \quad s \rightarrow \infty. 
\end{equation*}
This concludes the proof.
\end{proof}

\begin{lem}
\label{a-asymp-Q0}
Let $(f,a,b): [0,\infty) \rightarrow \R^3$ be a complete solution to the soliton equations (\ref{solitoneqn1})-(\ref{solitoneqn3}) with $Q_\infty = 0$ and $f''(0) < 0$. Then $a_{\infty} := \lim_{s\rightarrow \infty} a <\infty$ and $a$ is asymptotically constant.
\end{lem}
\begin{proof}
By Lemma \ref{b-asymp-Q0} above it follows that 
\begin{equation*}
\frac{b'}{b} \sim \frac{1}{2s} \quad \textrm{as} \quad s \rightarrow \infty
\end{equation*}
By Lemma \ref{QkeyLemma} we know $Q_{\infty} = 0$ implies $Q'(s) < 0$ for $s$ sufficiently large. From the evolution equation (\ref{Qeqn}) for $Q$
we then deduce that for any $\epsilon > 0$ there exists an $s_0 > 0$ such that for $s > s_0$
\begin{equation}
\label{ineq:Q}
Q'' + c_1 Q' \leq -\frac{c_2}{s} Q,
\end{equation}
where $c_1  = -f'_{\infty} + \epsilon$ and $c_2 = -\frac{1}{2} f'_{\infty} - \epsilon $. 
\vspace{1em}
\begin{claim}
For any $\epsilon >0$ there exist constants $C, s_0 >0$ such that for $s > s_0$
\begin{equation*}
Q(s) \leq C s^{-\frac{1}{2} + \epsilon}.
\end{equation*}
\end{claim}
\begin{claimproof}
Multiplying (\ref{ineq:Q}) by $e^{c_1 s}$ and integrating we obtain
\begin{align*}
Q'(s) &\leq e^{c_1\left(s_0-s\right)} Q'(s_0) - c_2 e^{-c_1 s} \int_{s_0}^s \frac{Q(t)}{t} e^{c_1 t}  \, \mathrm{d}t \\
	  &\leq e^{c_1\left(s_0-s\right)} Q'(s_0) - c_2 \frac{Q(s)}{s} e^{-c_1 s} \int_{s_0}^s  e^{c_1 t} \, \mathrm{d}t \\
	  & \leq e^{c_1(s_0-s)} Q'(s_0) + \left( - \frac{c_2}{c_1} + e^{c_1 \left(s_0 - s\right)} \right) \frac{Q(s)}{s} \\
	  & \leq \left(-\frac{c_2}{c_1}+\epsilon\right) \frac{Q(s)}{s},
\end{align*}
where by choosing $s_0$ sufficiently large we assumed that $Q'(s_0) < 0$ and $e^{c_1 \left(s_0 - s\right)} < \epsilon$.
Integrating this equation from $s_0$ to $s$ shows that 
\begin{equation*}
Q(s) \leq C s^{-\frac{c_2}{c_1} + \epsilon} \quad \text{for} \quad s > s_0,
\end{equation*}
where $C = Q(s_0)s_0^{\frac{c_2}{c_1}-\epsilon}$.
\end{claimproof}

This claim, the evolution equation (\ref{aQ}) of $a$ and the monotonicity properties of $a, b, f$, now imply that for any $\epsilon >0$ 
there exist constants $C_1, C_2, s_0>0$ such that for $s>s_0$
\begin{align*}
a''	&\leq \frac{2n}{a}Q^4 + a'f' \\  
    &\leq  \frac{C_1}{s^{2-\epsilon}} - C_2 a'.
\end{align*}
Integrating this differential inequality we see that
\begin{align}
\label{a-asymp-diff-ineq}
a'(s) &\leq a'(s_0) e^{C_2\left(s_0 - s\right)} + C_1 e^{-C_2s} \int_{s_0}^{s} \frac{e^{C_2 t}}{t^{2-\epsilon}} \: \mathrm{d}t.
\end{align}
An application of L'H\^opital's Rule shows
\begin{equation*}
C_1 e^{-C_2s} \int_{s_0}^{s} \frac{e^{C_2 t}}{t^{2-\epsilon}} \: \mathrm{d}t \sim \frac{C_1}{C_2} \frac{1}{s^{2-\epsilon}} \quad \textrm{as} \quad s \rightarrow \infty.
\end{equation*}
Integrating (\ref{a-asymp-diff-ineq}) we then see that $a$ is bounded.
\end{proof}

Finally, we prove that any complete solution to the soliton equations (\ref{solitoneqn1})-(\ref{solitoneqn3}) satisfies $Q \leq 1$ everywhere.
\begin{lem}
Let $(f,a,b): [0,s_0] \rightarrow \R^3$ be a solution to the soliton equations (\ref{solitoneqn1})-(\ref{solitoneqn3}) such that $Q(s_0) > 1$. Then the maximal extension of the solution $(f,a,b)$ blows up in finite distance $s$.
\end{lem}
\begin{proof}
Let $(f, a, b): [0, s_\infty) \rightarrow \R^3$, $s_\infty \in \R \cup \{ \infty\}$, be the maximal extension of the solution to the soliton equations (\ref{solitoneqn1})-(\ref{solitoneqn3}). Without loss of generality we may assume that $s_\infty = \infty$.

\vspace{1em}
\begin{claim}
There exists an $s_2 > 0$ such that $b' < 0$ for $s > s_2$.
\end{claim}
\vspace{0.5em}
\begin{claimproof}
By Lemma \ref{Qlim} we know that $Q$ is unbounded and hence there exists an $s_1>0$ such that
$$ Q^2 > n+ 2 \quad \text{for} \quad s > s_1.$$
It follows from (\ref{solitoneqn3}) and the monotonicity properties of $a$, $b$ and $f$ that for $s>s_1$ and as long as $b'>0$ we have
\begin{equation*}
b'' < - \frac{2}{b}.
\end{equation*}
Multiplying this equation by $b'$ and integrating, we see that there exists a $s_2 > s_1$ such that $b'(s_2) = 0$. As $Q^2 > n+2$, the proof of Lemma \ref{ab_mono} shows that $s_2$ is a strict maximum of $b$. Moreover, since $b'$ can only change its sign once, $b' < 0$ for $s > s_2$.
\end{claimproof}

This claim, the evolution equation (\ref{solitoneqn2}) of $a$, and the monotonicity properties of $a$ and $f$ show that 
\begin{equation*}
a'' \geq c_1 a^3 - c_2 a' \quad \text{for} \quad s > s_2,
\end{equation*}
where $c_1, c_2>0$ are some positive constants. This differential inequality forces $a$ to blow up in finite distance. We prove this using phase diagrams, as we did in the proof of Lemma \ref{asymptotics}: Write $z = a$ and $w = a'$ to obtain
\begin{align*}
z' &= w \\
w' &\geq c_2 \left( \frac{c_1}{c_2} z^3 - w\right) 
\end{align*}
Recall that $z' = a' > 0$ by Lemma \ref{ab_mono}. Therefore we can take $z$ to be the independent variable and deduce
\begin{align*}
\frac{dw}{dz} \geq c_2 \left( \frac{c_1}{c_2} \frac{z^3}{w} - 1\right).
\end{align*}
Set
\begin{equation*}
g(z) = \frac{c_1}{c_2} \frac{z^3}{z^{\frac{3}{2}}+1}
\end{equation*}
and consider the regions
\begin{align*}
 R_+&: w > g(z) \\
 R_-&: w < g(z)
 \end{align*}
 in the first quadrant $w, z > 0$ (where we take $w$ to be the $y$-axis and $z$ to be the $x$-axis). If $(z, w) \in R_-$ then
\begin{equation*}
\frac{dw}{dz} \geq c_2 z^{\frac{3}{2}}.
\end{equation*}
Integrating this differential inequality we see that $w(z)$ crosses over to the region $R_+$ in finite $z$. Notice that
on the curve $w = g(z)$ and for sufficiently large $z$ we have $\frac{dw}{dz} > g'(z)$. Thus $w(z)$ eventually remains in $R_+$. Switching back to the independent variable $s$ we see that for sufficiently large $s$
\begin{equation*}
a' \geq g(a).
\end{equation*}
Integrating shows that $a$ blows up in finite distance $s$.
\end{proof}

Theorem \ref{thm-asymptotics} follows from the lemmas above.

\section{Existence of non-collapsed complete solitons}
\label{section_existence}
So far we have only shown the existence of gradient steady solitons with $Q_{\infty} = 0$ (see Theorem \ref{thm-collapsed} and Remark \ref{rmk-collapsed}). These solitons are collapsed and therefore cannot occur as blowup limits of Ricci flow. In this section we construct a complete \emph{non-collapsed} steady soliton with $Q_{\infty} =1$ for $k>p$.

We begin by defining 
\begin{align*}
f^{\ast}_0 = \sup \{ f_0 \in \mathbb{R}\; |\; \text{for } f''(0) \leq f_0 \text{ a complete Ricci soliton exists} \}
\end{align*}
and noting that $f^{\ast}_0 > -\infty$ by Theorem \ref{thm-collapsed}. Recall that by Theorem \ref{thm-asymptotics} a complete solution to the soliton equations (\ref{solitoneqn1})-(\ref{solitoneqn3}) with $f''(0) < f^{\ast}_0$ satisfies $Q \leq 1$ everywhere. Below we show that $f^{\ast}_0 < 0$ for solitons with boundary conditions $a'(0) = \frac{k}{p}(n+1) > n+1$ and then argue that choosing $f''(0) = f^{\ast}_0$ leads to a complete steady gradient Ricci soliton with $Q_\infty =1$. Via the asymptotics for $a$ and $b$ when $Q_\infty=1$, as stated in Theorem \ref{thm-asymptotics}, it is straightforward to verify that such solitons are non-collapsed. 

\begin{lem}
\label{intersect}
Let $(f,a,b): [0, s_\infty) \rightarrow \R^3, s_\infty \in \R \cup \{ \infty \},$ be a maximal solution to the soliton equations (\ref{solitoneqn1})-(\ref{solitoneqn3}) with initial conditions $a'(0)>n+1$ and $f''(0) = 0$. Then $Q>1$ in finite distance $s$.
\end{lem}
\begin{proof}
Recall that $f = 0$ everywhere by Lemma \ref{f_mono}. By a change of variable of the form 
$$\frac{dr}{ds} = \frac{1}{p(r)}$$ 
for $p: (0, \infty) \rightarrow \R$ some positive function, the soliton equations (\ref{solitoneqn1})-(\ref{solitoneqn3}) can be written as
\begin{align}
\label{einstein1} 0 &= \frac{1}{a} \left(\frac{a'}{p}\right)' + 2n \frac{1}{b} \left(\frac{b'}{p}\right)' \\
\label{einstein2} \frac{1}{p} \left(\frac{a'}{p}\right)' &= 2n\left(\frac{a^3}{b^4} - \frac{a'b'}{bp^2}\right) \\
\label{einstein3} \frac{1}{p} \left(\frac{b'}{p}\right)' &= \frac{2n+2}{b}-2\frac{a^2}{b^3} - \frac{a'b'}{ap^2} - (2n-1)\frac{1}{b}\left(\frac{b'}{p}\right)^2.
\end{align}
Here $a,b,f,p$ are viewed as functions of $r$ and $'$ denotes differentiation with respect to $r$. These equations can be solved explicitly by taking the gauge 
\begin{equation*}
ap = L,
\end{equation*} 
for $L > 0$ a constant (see \cite{PP87}). Eliminating the term 
\begin{equation*}
\frac{1}{p} \left(\frac{a'}{p}\right)'
\end{equation*} 
in equation (\ref{einstein2}) by the expression obtained for it from (\ref{einstein1}) we deduce that
\begin{equation*}
b'' = - \frac{L^2}{b^3}.
\end{equation*}
One can check that this equation is solved by
\begin{equation}
\label{exp1}
b^2 = L^2-r^2.
\end{equation}
Substituting (\ref{exp1}) into (\ref{einstein3}) and applying the gauge condition $pa = L$ we obtain the first order equation 
\begin{equation*}
\left[\frac{a^2(r^2-L^2)^n}{(2n+2)L^2r} \right]' = - \frac{(r^2-L^2)^n}{r^2}.
\end{equation*}
Integrating yields an explicit solution of the form
\begin{align*}
a^2 &= -(2n+2)L^2r(r^2-L^2)^{-n}\int_{r_b}^r \frac{(s^2-L^2)^n}{s^2} \, \mathrm{d}s \\
b^2 &=(L^2-r^2) \\
p^2 &= \frac{L^2}{a^2},
\end{align*}
where $ -L < r_b < 0$ is some constant. A computation shows
\begin{align*}
\frac{da}{ds}\Big|_{s=0} &= \frac{1}{p(r)} \frac{da}{dr}\Big|_{r=r_b} = \frac{1}{2L} \frac{da^2}{dr}\Big|_{r=r_b} = -\frac{(n+1)L}{r_b} \\
b|_{s=0} &= (L^2 - r_b^2).
\end{align*}
Therefore taking 
\begin{align*}
L &= a'(0)\left( a'(0)^2- (n+1)^2 \right)^{-\frac{1}{2}}\\
r_b &= -\frac{(n+1)L}{a'(0)}
\end{align*}
we see that the solution satisfies the initial conditions $\frac{da}{ds} \big|_{s=0} = a'(0)$ and $b\big|_{s=0} = 1$. Taking the limit $r \rightarrow 0_-$ shows that 
\begin{align*}
a^2 = (2n+2) L^2 > L^2 = b^2
\end{align*}
at $r=0$, which proves the desired result.
\end{proof}

We now prove the existence of a non-collapsed steady Ricci soliton.

\begin{thm}
Let $(\hat{M}, J, \hat{g})$ be a K\"ahler-Einstein manifold of positive scalar curvature. Whenever $k > p(\hat{M}, \omega)$ there exists a non-collapsed steady gradient Ricci soliton on $L_{k}$ with $\lim_{s\rightarrow\infty} Q = 1$.
\end{thm}

\begin{proof}
First note that from the boundary condition $a'(0) = (n+1)\frac{k}{p} > n+1$ (see (\ref{boundaryconditions})) and Lemma \ref{intersect} it follows that $f^{\ast}_{0} < 0$. We proceed by proving
\vspace{1em}
\begin{claim}
 $f''(0)= f^{\ast}_0$ gives rise to a complete solution $(f, a, b): [0, \infty) \rightarrow \R^3$. 
\end{claim}
\vspace{0.5em}
\begin{claimproof}
We argue by contradiction and assume the contrary. By Lemma \ref{Qcontrol} it then follows that $Q$ becomes larger than $1$ after some finite distance $s$. By the continuous dependence on the initial condition $f''(0)$, however, the set
\begin{equation*}
\{f''(0) \in \mathbb{R} \;| \; Q > 1 \text{ after finite distance $s$} \}
\end{equation*}
is open. This contradicts the definition of $f^{\ast}_0$.
\end{claimproof}

 We now show that this solution satisfies $\lim_{s\rightarrow \infty} Q = 1$. Again we argue by contradiction and assume this were not the case. Then $\lim_{s\rightarrow \infty} Q = 0$ by Lemma \ref{Qlim}. From the monotonicity properties of $Q$ stated in Lemma \ref{QkeyLemma} it thus follows that there exists a unique $s_{\ast}>0$ at which $Q$ attains its maximum 
 $$Q_{max} := \max_{s\in[0,\infty)} Q = Q(s_{\ast}) < 1.$$ 
 Note that we cannot have $Q_{max} = 1$, as otherwise standard uniqueness results for ODEs would
 imply that $a = b$ everywhere. Fix an $s_{\ast\ast}> s_{\ast}$. 
 By the continuous dependence of the solution $(f,a,b)$ on $f''(0)$, we can find an $\epsilon_0 > 0$ 
 such that for all $\epsilon<\epsilon_0$ 
\begin{enumerate}
 \item a solution $(f_{\epsilon}, a_{\epsilon},b_{\epsilon}): [0,s_{\ast\ast}] \rightarrow \R^3$ 
 with $f_{\epsilon}''(0) = f^{\ast}_0 + \epsilon <0$ exists 

 \item $Q_{\epsilon} := \frac{a_{\epsilon}}{b_{\epsilon}}$ obtains 
 a local maximum $Q_{max,\epsilon} < 1$ at some $s_{\ast, \epsilon} \in (0,s_{\ast\ast})$. 
\end{enumerate}

From the monotonicity properties of $Q$ stated in Lemma \ref{QkeyLemma}, we deduce 
that $Q_{\epsilon} < 1$ on the maximal extensions of the solutions $(f_{\epsilon}, a_{\epsilon},b_{\epsilon})$, $\epsilon < \epsilon_0$. 
This, however, implies by Lemma \ref{Qcontrol} that $(f_{\epsilon}, a_{\epsilon},b_{\epsilon})$, $\epsilon < \epsilon_0,$ 
can be extended to complete solutions of the soliton equations (\ref{solitoneqn1})-(\ref{solitoneqn3}), contradicting the definition of $f^{\ast}_0$.
 
\end{proof}
The theorem above, in conjunction with Theorem \ref{asymptotics}, concludes our proof of the main Theorem \ref{main-thm}.

\section{Taub-Nut like solitons and the Bryant soliton}
As pointed out in the introduction, the completion of the warped product metric (\ref{warped}) on $\R_{>0}\times S^{2n+1}$ can be viewed as a metric on $\R^{2n+2}$, if at $s=0$ we choose the boundary conditions
\begin{align} 
\label{Rnboundary}
a&=b=0 \\ \nonumber
a'&= b'=1.
\end{align}
With these boundary conditions we see that the metric behaves like
\begin{equation*}
 g \sim ds^2 + s^2 g_{S^{2n+1}} \quad \text{as} \quad s \rightarrow 0.
\end{equation*}
To ensure smoothness at $s=0$ it is sufficient to require $a(s)$ and $b(s)$ to be extendable to smooth odd functions around $s=0$. Note that when $a=b$ everywhere, we obtain a rotationally symmetric metric 
\begin{equation}
\label{radsymmetric}
g = ds^2 + a(s)^2 g_{S^{2n+1}}
\end{equation}
and the soliton equations $(\ref{solitoneqn1})-(\ref{solitoneqn3})$ reduce to 
\begin{align}
\label{bryant1}
f'' &= (2n+1)\frac{a''}{a} \\
\label{bryant2}
a''&= \frac{2n}{a}\left(1-(a')^2\right) + a'f'.
\end{align}
These are the equations of a rotationally symmetric gradient steady soliton on $\R^{2n+2}$. This fact is exploited in Theorem \ref{r2n+2-thm} below, where we give another proof of the existence of the Bryant soliton in even dimensions greater than four.  


With boundary conditions (\ref{Rnboundary}) the soliton equations (\ref{solitoneqn1})-(\ref{solitoneqn3}) are, as previously, degenerate at $s=0$. Fortunately though, the proof of Theorem \ref{analiticity} carries over; it is straightforward to show that for every $a_0, b_0 \in \R$ there exists a unique analytic solution around $s=0$ satisfying $a'''(0)=a_0$ and $b'''(0)=b_0$. Moreover, the solution depends smoothly on $a_0$ and $b_0$.

Applying L'H\^opital's rule to equation (\ref{solitoneqn1}) we see that
\begin{equation*}
f''(0) = a'''(0) + 2nb'''(0)
\end{equation*}
and from (\ref{Rfirstorder}) it follows that 
$$R(0) = -2(n+1)f''(0).$$ 
Since $R \geq 0$ for any ancient solution to the Ricci flow it follows that $a'''(0) + 2nb'''(0) \leq 0$. From here on our previous results carry over word by word or with slight modifications, allowing us to prove the following theorem with little extra effort:

\begin{thm}
\label{r2n+2-thm}
Let $a_0 \leq b_0$ such that $f_0 = a_0 + 2n b_0 \leq 0$. Then there exists a complete solution $(f,a,b): [0,\infty) \rightarrow \R^3$ to the soliton equations (\ref{solitoneqn1})-(\ref{solitoneqn3}) with initial conditions $a=b=f=0$, $a'=b'=1$, $f'=0$, $f'' = f_0$, $a''' = a_0$ and $b''' = b_0$ at $s=0$. Furthermore there are three cases:
\begin{enumerate}
\item If $a_0 +2n b_0 = 0$ and $a_0 = b_0$, we obtain the standard Euclidean metric.
\item If $a_0 +2n b_0 = 0$ and $a_0 < b_0$, we obtain a Taub-Nut metric with asymptotics $a \sim const$ and $b \sim s$.
\item If $a_0 +2n b_0 < 0$ and $a_0 = b_0$, we obtain the Bryant soliton with asymptotics $a = b \sim const \sqrt{s}$.
\item If $a_0 +2n b_0 < 0$ and $a_0 < b_0$, we obtain a Taub-Nut like Ricci soliton with asymptotics $a \sim const$ and $b \sim const \sqrt{s}$.
\end{enumerate}
\end{thm}

\begin{proof}
As explained at the beginning of this section there exists an analytical solution around $s=0$ to the soliton equations (\ref{solitoneqn1})-(\ref{solitoneqn3}) with boundary conditions (\ref{Rnboundary}). Below we show that in each of the cases this local solution can be extended to a complete solution. 

In cases (1) and (2) we have $f''(0) = 0$ and hence $f\equiv0$ everywhere by Lemma \ref{f_mono}. It easily seen that $a = b = s$ is the unique solution in case (1) and that it corresponds to the standard Euclidean metric on $\R^{2n+2}$. In case (2) we obtain the Taub-Nut metrics as derived in \cite{AG03}. 

For the remaining cases note that by L'H\^opital's rule 
\begin{align*}
\lim_{s \rightarrow 0} Q  &= 1 \\
\lim_{s \rightarrow 0} Q'  &= 0 \\
\lim_{s \rightarrow 0} Q''  &= a'''(0) - b'''(0)
\end{align*}
Hence in case (4) we see that $Q'<0$ for small $s>0$. By Lemma \ref{QkeyLemma} it follows that $Q'<0$ for as long as the solution exists and thus by Lemma \ref{Qcontrol} we obtain a complete Ricci soliton $(f,a,b): [0,\infty) \rightarrow \R^3$. From Lemma \ref{Qlim} it follows that $\lim_{s\rightarrow\infty} Q = 0$ and therefore $a \sim const$ and $b\sim const \sqrt{s}$ as $s\rightarrow \infty$ by Lemma \ref{a-asymp-Q0} and Lemma \ref{b-asymp-Q0}.

For case (3) we need to prove that $Q=1$ everywhere. We argue by contradiction. Assume there exists an $s_0 >0$ such that $Q(s_0) > 1$. By the continuous dependence on boundary conditions, we can pick an $\epsilon>0$ sufficiently small such that the solution $(f_{\epsilon}, a_{\epsilon}, b_{\epsilon}): [0,s_0] \rightarrow \R^3$ satisfying boundary conditions $a'''(0) = a_0 - \epsilon$ and $b'''(0) = b_0$ exists and $Q_{\epsilon}(s_0):=\frac{a_{\epsilon}(s_0)}{b_{\epsilon}(s_0)} > 1$. This, however, contradicts that $Q' \leq 0$ and $Q \leq 1$ everywhere in case (4). Therefore $Q\leq1$ everywhere and by Lemma \ref{Qcontrol} we obtain a complete solution $(f,a,b): [0,\infty) \rightarrow \R^3$. Now assume that there exists an $s_0$ such that $Q(s_0) < 1$. Then we can choose an $\epsilon >0$ such that the solution $(f_{\epsilon}, a_{\epsilon}, b_{\epsilon}): [0,s_0] \rightarrow \R^3$ with boundary conditions $a'''(0) = a_0 + \epsilon$ and $b'''(0) = b_0$ exists and $Q_{\epsilon}(s_0) < 1$. This, however, leads to a contradiction, as $Q''_{\epsilon}(0) > 0$ and thus $Q'_{\epsilon}(s) >0$ for $s>0$ by Lemma \ref{QkeyLemma}. We conclude that $Q=1$ and hence $a=b$ everywhere. As the soliton equations simplify to the rotationally symmetric equations (\ref{bryant1}) and (\ref{bryant2}) when $Q=1$, the solution must be homothetic to the Bryant soliton. The asymptotics of the Bryant soliton follow from Corollary \ref{Qlim1-asymptotic}.
\end{proof}


\section{Conjectures}
In this section we briefly discuss two conjectures relating to the non-collapsed solitons of Theorem \ref{main-thm}. We numerically integrated the soliton equations $(\ref{solitoneqn1})-(\ref{solitoneqn3})$ and found strong support for the following conjecture:
\begin{conj}
On a line bundle $L_k$, $k > p ( \hat{M}, \omega)$, the complete non-collapsed steady gradient soliton of Theorem \ref{main-thm} is unique up to scaling and isometry in the class of metrics (\ref{metric}). Moreover, choosing the normalization $b(0)=1$, there exists an $f^{\ast}_0 < 0$ such that the maximum
solution $(f, a, b): [0, s_\infty) \rightarrow \R^3, s_\infty \in \R \cup \{ \infty\},$ to the soliton equations (\ref{solitoneqn1})-(\ref{solitoneqn3}) is
\begin{enumerate}
\item incomplete when $f''(0) > f^{\ast}_0$ 
\item complete and non-collapsed when $f''(0) = f^{\ast}_0$
\item complete and collapsed when $f''(0) < f^{\ast}_0$ 
\end{enumerate}
\end{conj}

Motivated by the discovery of the non-collapsed steady solitons in this paper, we conducted Ricci flow simulations for $U(2)$-invariant metrics of the form (\ref{U2-inv-metric}) in four dimensions to investigate whether the solitons of Corollary \ref{cor:4d-non-collapsed-solitons} appear as blow-up limits of singularities. Our results indicate that they do indeed and a paper is in preparation. We therefore conjecture:

\begin{conj}
The non-collapsed steady solitons of Theorem \ref{main-thm} all occur as singularity models in Ricci flow.
\end{conj}

Note that in the case of $L_1$ over $\mathbb{C}P^1$ (i.e. $n=k=1$ and $p=2$ in our notation), Davi Maximo already showed in \cite{M14} that the FIK shrinker, which is the unique shrinking K\"ahler-Ricci soliton for a metric of the form (\ref{metric}), occurs as a singularity model.

In Figures 1 and 2 examples of complete solitons with $Q_{\infty} = 0$ and $Q_{\infty} = 1$, respectively, are depicted. 

\begin{figure}[h]
\centering
\begin{minipage}{.5\textwidth}
\label{fig1}
\centering
\includegraphics[width=\linewidth]{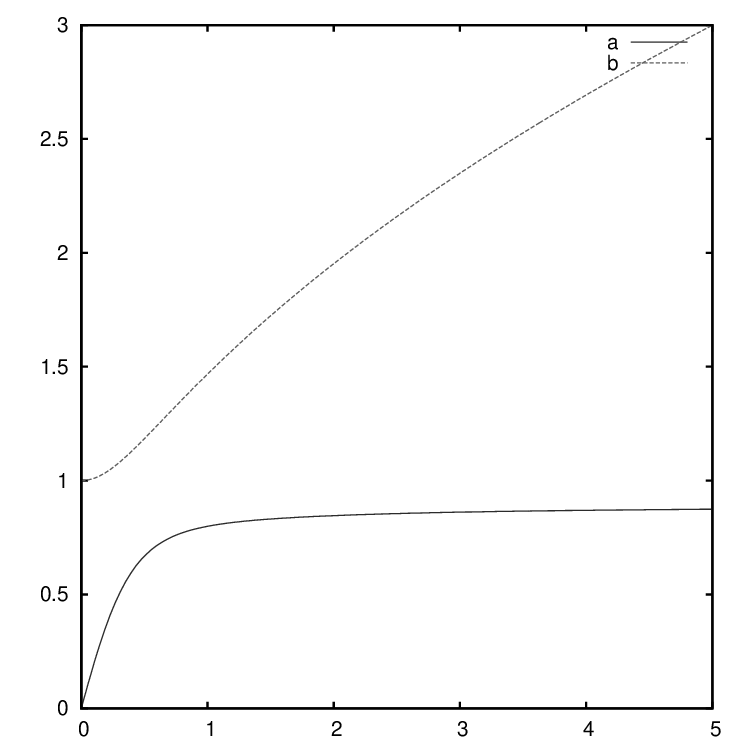}
\caption{A collapsed soliton on $\R_{>0}\times S^3/\mathbb{Z}_2$ with $f''(0) = -10$}
\end{minipage}%
\begin{minipage}{.5\textwidth}
\label{fig2}
\centering
\includegraphics[width=\linewidth]{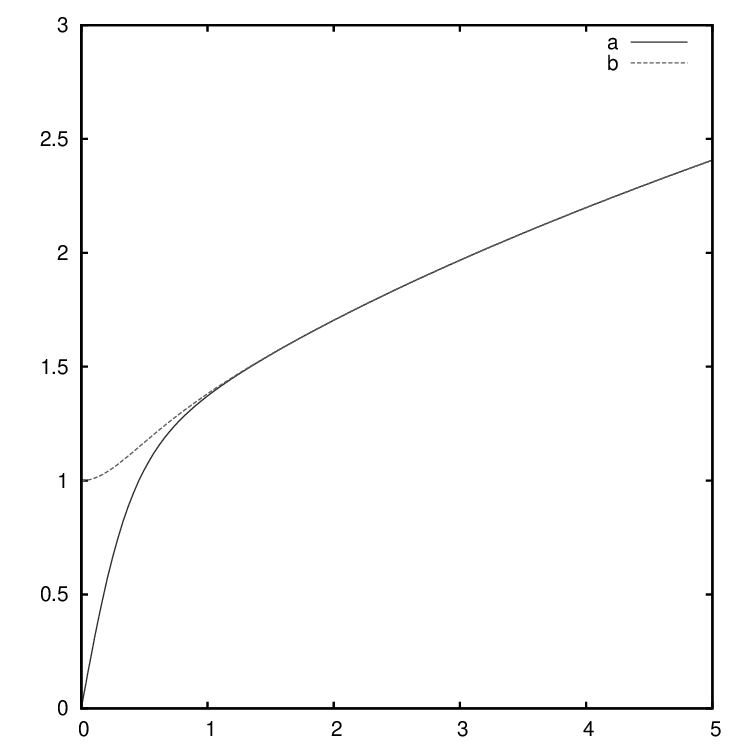}
\caption{The non-collapsed soliton on $\R_{>0}\times S^3/\mathbb{Z}_3$}
\end{minipage}
\end{figure}

\section*{Acknowledgments}
The author thanks his PhD advisors Prof. Richard Bamler and Prof. Jon Wilkening for their generosity in time and advice, without which the solitons would have not been found. The author would also like to point out that the proof of Lemma \ref{asymptotics} is due to Prof. Jon Wilkening. We thank Yongjia Zhang for pointing out an error in the Appendix and a gap in Section 6. Finally we thank the anonymous referee for his/her helpful comments.

This work was supported in part by the U.S. Department of Energy, Office of Science, Applied Scientific Computing Research, under award DE-AC02-05CH11231.

\section*{Appendix A}
Here we derive the Ricci soliton equations. We will follow \cite{PP87} to compute the Ricci tensor of the metric
\begin{equation*}
g = ds^2 + a(s)^2 \left( d\tau - 2 A\right)^2 + b(s)^2 \hat{g},
\end{equation*}
on a complex line bundle of a K\"ahler-Einstein manifold $(\hat{M}^{2n}, J, \hat{g})$, where $A$ is a connection 1-form on $\hat{M}$ such that $dA = \omega$ and $\omega$ is K\"ahler form of $\hat{M}$. We will assume that the metric $\hat{g}$ is scaled such that $Ric(\hat{g}) = 2(n+1) \hat{g}$, in order for $a$ and $b$ to have a nice geometrical interpretation when we choose $\mathbb{C}P^n$ equipped with the Fubini-Study metric as the base manifold. For the same reason we multiply the  connection form $A$ by 2.

We will compute the full curvature tensor of $g$ using Cartan's formalism. Pick an orthonormal frame of 1-forms $e^0 = ds$, $e^1 = a\left(d\tau - 2A\right)$ and $e^i = b \hat{e}^i$, $i = 2, 3, \cdots, 2n+1$, where $\hat{e}^i$ is an orthonormal frame on the base $\hat{M}$. Denote by $e_i$, $i = 0, 1, \cdots, 2n+1$ and $\hat{e}_i$ , $i = 2, 3, \cdots 2n+1$ the corresponding dual basis. In the following indices will run from either $0$ to $2n+1$ or $2$ to $2n+1$, which will be clear from context. 

The connection 1-forms $\theta^i_j$, defined by $\nabla e_i = \theta^j_i e_j$, and the curvature 2-forms $\Omega^j_i$, defined by $R(\cdot, \cdot)e_i = \Omega^j_i e_j$, satisfy Cartan's structure equations
\begin{align*}
de^i &= - \theta^i_j \wedge e^j \\
\theta^i_j &= - \theta^j_i\\
\Omega^i_j &= d\theta^i_j + \theta^i_k \wedge \theta^k_j.
\end{align*}
Note that in coordinates we have $\Omega^i_j = \frac{1}{2} R^i_{\;jkm} e^k\wedge e^m$. In the following we will denote by $\hat{\theta}^i_j$ and $\hat{\Omega}^i_j$ the connection 1-forms and curvature 2-forms respectively, corresponding to the frame $\hat{e}^i$ on the base $(\hat{M}, \hat{g})$. Moreover $\hat{\nabla}$ will be the covariant derivative on $(\hat{M}, \hat{g})$. Hence we compute
\begin{align*}
\theta^0_1 &= -\frac{a'}{a} e^1 && \theta^0_i = -\frac{b'}{b} e^i \\
\theta^1_i &= - \frac{a}{b^2} \omega_{ij} e^j && \theta^i_j = \hat{\theta}^i_j + \frac{a}{b^2} \omega_{ij} e^1.
\end{align*}
Proceeding, we obtain
\begin{align*}
\nonumber \Omega^0_1 &= -\frac{a''}{a} e^0 \wedge e^1 + \frac{1}{b^2}\left(a' - \frac{ab'}{b}\right) w_{ij} e^i\wedge e^j \\
\nonumber \Omega^0_i &= - \frac{b''}{b} e^0 \wedge e^i + \frac{1}{b^2}\left(a' - \frac{ab'}{b}\right) w_{ij}e^1 \wedge e^j \\
\nonumber\Omega^1_i &= \left(\frac{a^2}{b^4} \omega_{kj}\omega_{ki} - \delta_{ij} \frac{a'b'}{ab}\right) e^1\wedge e^j - \frac{1}{b^2} \left(a' - \frac{ab'}{b}\right) \omega_{ij} e^0\wedge e^j \\
\nonumber\Omega^i_j &= \hat{\Omega}^i_j - \left(\frac{b'}{b}\right)^2 e^i \wedge e^j - \frac{a^2}{b^4}\left( \omega_{ij}\omega_{km} + \omega_{ik} \omega_{jm} \right) e^k \wedge e^m +\frac{2}{b^2}\left( a' - \frac{ab'}{b} \right)\omega_{ij} e^0 \wedge e^1 
\end{align*}
Note that we used that the complex structure $J$ is parallel for a K\"ahler manifold and thus $\omega_{ik}\hat{\theta}^k_j = \hat{\theta}^k_i \omega_{kj}$. Finally we can compute the non-zero entries of the Ricci tensor via $R_{ij} = R^k_{\;ikj}$
\begin{align*}
R_{00} &= - \frac{a''}{a} - 2n \frac{b''}{b} \\
R_{11} &= - \frac{a''}{a} + 2n\left(\frac{a^2}{b^4} - \frac{a'b'}{ab}\right) \\
R_{ii} &= - \frac{b''}{b} + \frac{2n+2}{b^2}- 2\frac{a^2}{b^4} - \frac{a'b'}{ab} - (2n-1)\left(\frac{b'}{b}\right)^2 
\end{align*}
From this we can also compute the scalar curvature
\begin{equation}
\label{scalarcurvature}
R = -2 \frac{a''}{a} - 4n \frac{b''}{b} - 4n \frac{a'b'}{ab} - 2n(2n-1)\left(\frac{b'}{b}\right)^2 - 2n \frac{a^2}{b^4} + \frac{2n(2n+2)}{b^2}
\end{equation}
Finally we need to compute the Hessian $\nabla^2f$. From Koszul's formula it follows that the only non-zero terms are
\begin{align*}
\nabla^2_{e_0,e_0} &= f'' \\
\nabla^2_{e_1,e_1} &= \frac{a'}{a}f' \\
\nabla^2_{e_i,e_i} &= \frac{b'}{b}f'
\end{align*}
Therefore we obtain the soliton equations (\ref{solitoneqn1})-(\ref{solitoneqn3}). From above it also follows that the Laplacian $\Delta \Phi$ of a function $\Phi: M \rightarrow \R$ depending only on $s$ can be written as
\begin{equation}
\label{laplacian}
\Delta \Phi = \Phi'' + \left(\frac{a'}{a}+ 2n \frac{b'}{b}\right)\Phi'.
\end{equation}
\section*{Appendix B}
Here we prove Theorem \ref{analiticity} ascertaining the local existence of analytic solutions to the soliton equations (\ref{solitoneqn1})-(\ref{solitoneqn3}) around the origin. We begin by proving Theorem \ref{BBgen}, which generalizes the following result of the French mathematicians Briot and Bouquet to a parameter dependent system of ODEs:
\begin{thm}[Briot and Bouquet 1856, \cite{BB1856}]
Let $f: \R\times\R \rightarrow \R$ be an analytic function vanishing at $(0,0)$ and its derivative $\frac{\partial f}{\partial u}(0,0)$ not be a positive integer. Then there exists an  analytical solution $u$ around $r=0$ to the non-linear ODE
\begin{equation*}
r\frac{du}{dr} = f(u,r).
\end{equation*}
\end{thm}
We then show how the soliton equations (\ref{solitoneqn1})-(\ref{solitoneqn3}) can be put in a form such that Theorem \ref{BBgen} can be applied. This yields the proof of Theorem \ref{analiticity}. 

\begin{thm}
\label{BBgen}
Let $n \in \N$, $c \in \R$ and $U \subset \R^n$ an open subset containing the origin. Let 
\begin{align*}
P: \quad  U \times \R \times \R &\longrightarrow \R^n  \\ 
 	 (u, r, \lambda) &\longrightarrow P(u,r,\lambda)
\end{align*}
be a vector valued analytic function around $(\vec{0}, 0 , c)$ such that $P(\vec{0},0,\lambda) = 0$ for all $\lambda \in \R$. If there is an open interval $I \ni c$ such that for all $\lambda \in I$ $\frac{\partial P}{\partial u}(\vec{0},0, \lambda)$ has no positive integer eigenvalues and
\begin{equation*}
B = \sup\limits_{\substack{\lambda \in I \\ m \in \N} } \norm{ \left(m I_n - \frac{\partial P}{\partial u}(\vec{0},0,\lambda)\right)^{-1} } < \infty,
\end{equation*}
then there exists an $\epsilon > 0$ and a one-parameter family of analytic vector valued functions $u(\cdot, \lambda): (-\epsilon, \epsilon) \rightarrow \R^n$ solving the ODE system
\begin{align}
\label{bb-eqn}
r \frac{d u(r, \lambda) }{d r} &= P(u(r,\lambda),r, \lambda)\\ \nonumber
u(0, \lambda) &= 0
\end{align}
for  $\lambda \in (c- \epsilon, c+ \epsilon)$. Furthermore $u$ depends analytically on $\lambda$.
\end{thm}
\begin{remark}
\begin{enumerate}
\item $I_n$ denotes the $n\times n$ identity matrix.
\item For a matrix $M$ we denote by $\norm{M}$ the operator norm with respect to the standard Euclidean norm on $\R^n$.
\end{enumerate}
\end{remark}

\begin{proof}
We follow the proof of the one-dimensional case presented in \cite{H79}[Theorem 11.1]. Denote by $u_i$ and $P_i$, $i = 1, 2, \cdots, n$, the components of $u$ and $P$ respectively. By analyticity we can write $P_i$ as a power series around the origin
\begin{equation*}
P_i(u,r, \lambda) = \sum_{k_1, \cdots, k_{n+2} \in \N} P_{i k_1 \cdots k_{n+2}} u_1^{k_1} \cdots u_n^{k_n} r^{k_{n+1}} \left(\lambda - c \right)^{k_{n+2}} 
\end{equation*}
such that for some $M > 0$, $R>0$ 
\begin{equation*}
|P_{i k_1 \cdots k_{n+2}}| < \frac{M}{R^{k_1 + \cdots + k_{n+2}}}
\end{equation*}
for $i = 1, 2, \cdots, n$ and $k_1, \cdots, k_{n+2} \in \N_{0}$. That is to say the power series converges whenever $|u_1|, \cdots, |u_n|, |r|, |\lambda -c| < R$. Defining the analytic functions
\begin{equation*}
c_{i k_1 \cdots k_{n+1}}(\lambda) := \sum_{k_{n+2} \in \N_0} P_{i k_1 \cdots k_{n+2}}(\lambda-c)^{k_{n+2}}
\end{equation*}
we have that for $|\lambda - c| < \frac{R}{2}$ 
\begin{equation*}
 \left|c_{i k_1 \cdots k_{n+1}}(\lambda)\right| < \frac{2M}{R^{k_1 + \cdots k_{n+1}}}.
\end{equation*}
Letting
\begin{align*}
c_i(\lambda) = \frac{\partial P_i}{\partial r} (\vec{0},0, \lambda) \\
c_{ij}(\lambda)  = \frac{\partial P_i}{\partial u_j} (\vec{0},0, \lambda)
\end{align*}
we can then write
\begin{equation*}
P_i(\vec{u},r, \lambda) = c_i(\lambda) r + \sum_{j= 1}^n c_{ij}(\lambda) u_{j} + \sum_{k_1 + \cdots + k_{n+1} \geq 2} c_{i k_1 \cdots k_{n+1}}(\lambda) u_1^{k_1} \cdots u_n^{k_n} r^{k_{n+1}}
\end{equation*}
for $i = 1, 2, \cdots, n$, whenever $|u_i|, r < R$ and $|\lambda - c| < \frac{R}{2}$. Below we fix such a $\lambda$ and omit stating the dependence of our quantities on it. 

We proceed by constructing a formal power series solution of the form 
\begin{equation}
\label{ansatz}
u_i(r) = \sum_{j = 1}^{\infty} a_{ij} r^j
\end{equation} 
for $i = 1 ,2, \cdots, n$ and $a_{ij} \in \R$. By substituting (\ref{ansatz}) into (\ref{bb-eqn}) we obtain
\begin{align*}
\sum_{j = 1}^{\infty} j a_{ij} r^j &= c_i r + \sum_{j= 1}^n c_{ij}\left( \sum_{q = 1}^{\infty} a_{jq} r^q \right)  \\ 
			& \quad  + \sum_{k_1 + \cdots + k_{n+1} \geq 2} c_{i k_1 \cdots k_{n+1}} \left(\sum_{j = 1}^{\infty} a_{k_1 j} r^j\right)^{k_1} \cdots \left(\sum_{j = 1}^{\infty} a_{k_n j} r^j\right)^{k_n} r^{k_{n+1}}
\end{align*}
for $i = 1, 2, \cdots n$. By expanding and collecting terms of equal order we deduce that
\begin{equation*}
\sum_{k = 1}^n \left(\delta_{ik} - c_{ik}\right) a_{k1} = c_i 
\end{equation*}
for the first order terms and
\begin{equation*}
\sum_{k = 1}^n\left(j \delta_{ik} - c_{ik}\right) a_{kj} = M_j(c_{ik_1k_2 \cdots k_{n+1}}; \{a_{pq}\; |\; q \leq j-1,  1 \leq p \leq n \} )
\end{equation*}
for the $j$-th order terms ($j > 1$), where $M_j$ is a multinomial with \emph{non-negative} coefficients depending on the variables indicated. In the following denote by $D(m)$, $m \in \N$ the matrix with components
\begin{equation*}
m \delta_{ij} - c_{ij}.
\end{equation*}
Because the matrix $c_{ij}$ has no positive integer eigenvalues, $D(m)$ is invertible for all $m\in \N$ and we can uniquely determine $a_{ij}$ order by order. In the following we will show that the resulting power series (\ref{ansatz}) has a positive radius of convergence. 

For this consider an analytic vector valued function 
\begin{equation*}
G: \R^n \times \R \rightarrow \R^n
\end{equation*}
given by
\begin{equation*}
G_i(\vec{u},r) = C_i r + \sum_{k_1 + \cdots + k_{n+1} \geq 2} C_{i k_1 \cdots k_{n+1}} u_1^{k_1} \cdots u_n^{k_n} r^{k_{n+1}}
\end{equation*}
that majorizes $P$ for all non-first order terms in $u_i$
\begin{align*}
|c_i| &\leq C_i  \\
|c_{i k_1 \cdots k_{n+1}}| &\leq  C_{i k_1 \cdots k_{n+1}}
\end{align*}
and for which the Jacobian vanishes
\begin{equation*}
\frac{\partial G}{\partial u} (\vec{0},0) = \vec{0}.
\end{equation*}
We choose
\begin{equation*}
G_i(\vec{u},r) = \frac{2M}{\left(1- \frac{r}{R} \right)\left( 1 - \frac{1}{R}\left( u_1 + \cdots + u_n \right) \right)} - 2M \left( 1 + \frac{1}{R}\left( u_1 + \cdots + u_n \right) \right)
\end{equation*}
in which case
\begin{align*}
C_1 &= C_2 = \cdots = C_n \\
C_{1 k_1 \cdots k_{n+1}} &= C_{2 k_1 \cdots k_{n+1}} = \cdots = C_{n k_1 \cdots k_{n+1}}
\end{align*}
for $k_1, \cdots, k_{n+1} \in \N$. We proceed by finding an analytic function
\begin{align*}
Y:\; \R &\longrightarrow \R^n \\
Y(0)&= \vec{0}
\end{align*} 
solving the implicit equation 
\begin{equation}
\label{implicit-eqn}
Y_i(r) =  B\sqrt{n} G_i( Y(r), r), \; i = 1, 2, \cdots, n.
\end{equation}
and show that it majorizes the formal power series solution found for $u$ above, thereby proving the desired result. For our choice of $G$ the equation (\ref{implicit-eqn}) is quadratic and solved by
\begin{equation*}
Y_i(r) = \frac{1 - \sqrt{ 1 - 8\sqrt{n}MB \left( 2\sqrt{n}MB\left(\frac{n}{R}\right)^2 + \frac{n}{R} \right)\left(\frac{r}{r-R}\right)  } }{4\sqrt{n}MB\left(\frac{n}{R}\right)^2 + 2\frac{n}{R}}
\end{equation*}
for $i = 1, 2, \cdots, n$. Note that $Y_i$ vanishes at the origin and is analytic around $r =0$ with radius of convergence
\begin{equation*}
R_c = \frac{R}{1+ 8 \sqrt{n}MB \left(2 \sqrt{n} MB \left(\frac{n}{R}\right)^2 + \frac{n}{R}\right)} > 0.
\end{equation*}
Therefore we can write $Y$ as a power series
\begin{equation*}
Y_i(r) = \sum_{j = 1}^{\infty} A_{ij} r^j
\end{equation*}
for $i = 1, 2, \cdots, n$ and $|r| < R_c$. Because the $Y_i$, $i = 1, 2, \cdots, n$, are all equal we have
\begin{equation*}
A_{1j} = A_{2j} = \cdots = A_{nj}
\end{equation*}
for $j \in \N$. Note that we can compute the $A_{ij}$ by solving the implicit equation (\ref{implicit-eqn}) order by order. This leads to
\begin{equation*}
A_{i1} = B\sqrt{n} C_{i}
\end{equation*}
and
\begin{equation*}
A_{ij} = B\sqrt{n} M_j(C_{ik_1k_2 \cdots k_{n+1}}; \{A_{pq}\; |\; q \leq j-1,  1 \leq p \leq n \} )
\end{equation*}
for $j >1$ and $i = 1, 2, \cdots, n$, where $M_j$ is the same multinomial as above. This allows us to show by induction on $j$ that
\begin{equation*}
|a_{ij}| \leq A_{ij}
\end{equation*}
for $i = 1, 2, \cdots, n$ and $j \in \N$. In particular, notice that 
\begin{equation*}
|a_{i1}| = |\sum_{j=1}^n D(1)^{-1}_{ij} c_{j}| < B \sqrt{n}C_{1} = A_{i1},
\end{equation*}
where we used the assumption that
\begin{equation*}
\norm{D^{-1}(m)} \leq B
\end{equation*}
for $m\in \N$. By induction
\begin{align*}
|a_{ij}| &=  \left| \sum_{q=1}^n D(j)^{-1}_{iq}  M_j(c_{qk_1k_2 \cdots k_{n+1}}; \{a_{pq}\; |\; q \leq j-1,  1 \leq p \leq n \})\right| \\
		 &\leq  B \sqrt{n} M_j(C_{1k_1k_2 \cdots k_{n+1}}; \{A_{pq}\; |\; q \leq j-1,  1 \leq p \leq n \} ) \\
		 & = A_{ij}
\end{align*}
Hence the formal power series solution for $u$ converges with radius of convergence greater or equal to $R_c$. Because $R_c$ does not depend on $\lambda$ as long as $|\lambda - c| \leq \frac{R}{2}$ and the coefficients $a_{ij}$ depend analytically on $\lambda$ we deduce that $u$ varies analytically with $\lambda$. 
\end{proof}

Now we can prove Theorem \ref{analiticity}:

\begin{proof}[Proof of Theorem \ref{analiticity}]
Since $a'(s) \neq 0$, locally at $s=0$ we can take $a$ as the independent variable of the soliton equations (\ref{solitoneqn1})-(\ref{solitoneqn3}) by considering the following change of variables 
\begin{equation*}
g = \frac{da^2}{h(a^2)} + g_{a,b(a)}.
\end{equation*}
Therefore taking $r = a^2$, we have
\begin{equation*}
\frac{dr}{ds} = 2 \sqrt{r h(r)}
\end{equation*}
and if we write $\dot{b}$ for $\frac{\partial b}{\partial r}$ etc. our soliton equations read
\begin{align}
\label{appeqn1}\ddot{f} &= \frac{1}{4r}\frac{\dot{h}}{h} + 2n \frac{\ddot{b}}{b} + \frac{n}{r}\frac{\dot{b}}{b} + n \frac{\dot{h}\dot{b}}{h b}  - \frac{1}{2r} \dot{f} - \frac{1}{2} \frac{\dot{h}}{h} \dot{f}\\
\dot{h} &= \frac{2n r}{ b^4} - 4nh \frac{\dot{b}}{b} + 2h \dot{f}\\
\label{appeqn2}\ddot{b} &= \frac{n+1}{2rhb} - \frac{1}{2hb^3} - \frac{\dot{b}}{r} - \frac{1}{2} \frac{\dot{h}}{h}\dot{b} - (2n-1) \frac{\dot{b}^2}{b} + \dot{f}\dot{b}
\end{align}
with boundary conditions 
\begin{align*}
b(0)&=1 \\
\dot{b}(0) &= \frac{n+1}{2a_0^2}\\
h(0) &= a_0^2\\
f(0) &=0\\
\dot{f}(0) &= \frac{f''(0)}{2 a_0^2} \equiv c
\end{align*}
Note that for fixed $n \in \N$ and $a_0\in\R_{>0}$ we can freely vary $\dot{f}(0) = c$. The boundary condition for $\dot{b}$ was derived by using the L'H\^opital's Rule and noting that (\ref{solitoneqn3}) at $s=0$ implies that $b''(0)=n+1$. Since only $\dot{f}$ and $\ddot{f}$ appear in the equation we may consider this ODE as first order in $\dot{f}$. Furthermore, defining $F = \dot{f}$ and $B = \dot{b}$ we can turn the equations (\ref{appeqn1})-(\ref{appeqn2}) into a first order system of ODEs in $(F,h,b,B)$ 
\begin{align*}
r\dot{F} &= -F^2r + 4nr \frac{FB}{b} - 2n\frac{B}{b} - 2n(2n-1)r \frac{B^2}{b^2} \\
\nonumber& \qquad\qquad\qquad\qquad\qquad\qquad+ \frac{n(n+1)}{hb^2} - \frac{n}{2hb^4}\left(r + 2 F r^2 \right) \\
r\dot{h} &= \frac{2nr^2}{b^4} - 4nhr \frac{B}{b} + 2hrF \\
r\dot{b}&= Br \\
r\dot{B} &= \frac{n+1}{2hb} - \frac{r}{2hb^3} - B - \frac{Br^2n}{hb^4} + \frac{rB^2}{b}
\end{align*}
Defining $u(\cdot, c) = (u_1(\cdot,c),u_2(\cdot,c),u_3(\cdot,c),u_4(\cdot,c)) \equiv (F(\cdot)-c,h(\cdot)-h(0),b(\cdot)-b(0),B(\cdot)-B(0))$ for $c\in \mathbb{R}$ we obtain an ODE system with parameter $c$ of the form 
\begin{align}
\label{scheme}
r \frac{d u_i}{d r} &= P_i(u,r,c) \\ \nonumber
	u_i(0,c) & = 0 \quad \text{for }i = 1, 2, 3, 4,
\end{align}
where $P$ is an analytic function in the neighborhood of the point $(\vec{0},0, c)$ in $\mathbb{C}^6$ and $P(\vec{0},0,c) = 0$. We compute $\frac{\partial P_i} {\partial u_j}$ at $(\vec{0},0,c)$ and obtain
\begin{equation*}
\begin{bmatrix}
    0 & -\frac{n(n+1)}{a_0^4} & -\frac{n(n+1)}{a_0^2} & -2n \\
    0 & 0&0&0 \\
    0 & 0&0&0 \\
    0 & -\frac{(n+1)}{2a_0^4} & -\frac{(n+1)}{2a_0^2} & -1 \\
\end{bmatrix}.
\end{equation*}
This matrix has characteristic polynomial 
\begin{equation*}
\text{det}(m I - \frac{\partial P}{\partial u}) = m^3(m+1),
\end{equation*}
which has no positive integer roots. Therefore the inverse
\begin{equation}
\label{invmatrix}
\left(m I - \frac{\partial P}{\partial u}\right)^{-1} = \begin{bmatrix}
 \frac{1}{m} & -\frac{n (n+1)}{a_0^4 m (m+1)} & -\frac{n (n+1)}{a_0^2 m (m+1)} & -\frac{2
   n}{m^2+m} \\
 0 & \frac{1}{m} & 0 & 0 \\
 0 & 0 & \frac{1}{m} & 0 \\
 0 & -\frac{n+1}{2 a_0^4 m (m+1)} & -\frac{n+1}{2 a_0^2 m (m+1)} & \frac{1}{m+1} \\
 \end{bmatrix}
\end{equation}
exists for $m\in \N$. Furthermore we can find a $B\in \R$ such that 
\begin{equation*}
\norm{\left(m I - \frac{\partial P}{\partial u}\right)^{-1}} < B
\end{equation*}
for all $m\in\N$. Therefore we can apply Theorem \ref{BBgen} proving the desired result. 
\end{proof}

\end{document}